\documentclass[11pt]{amsart}

\usepackage{amsmath}
\usepackage{amssymb}
\usepackage{amscd}
\usepackage{color}
\usepackage{hyperref}
\usepackage{mathrsfs}
\usepackage{eucal}
\usepackage{upgreek}
\usepackage[makeroom]{cancel}
\usepackage[normalem]{ulem}
\usepackage{array}
\usepackage{verbatim}
\usepackage{mathtools}

\usepackage{enumitem}

\usepackage{xy}
\xyoption{all}
\usepackage{tikz-cd}

\newcommand{\nc}{\newcommand}

\nc{\CC}{{\mathbb{C}}}
\nc{\DD}{{\mathbb{D}}}
\nc{\LL}{{\mathbf{L}}}
\nc{\RR}{{\mathbf{R}}}
\renewcommand{\P}{{\mathbb{P}}}
\nc{\OO}{{\mathbb{O}}}

\nc{\QQ}{{\mathbb{Q}}}
\nc{\ZZ}{{\mathbb{Z}}}
\nc{\Z}{{\mathbb{Z}}}

\nc{\cA}{{\mathcal{A}}}
\nc{\cB}{{\mathcal{B}}}
\nc{\cC}{{\mathcal{C}}}
\nc{\cD}{{\mathcal{D}}}
\nc{\cE}{{\mathcal{E}}}
\nc{\cF}{{\mathcal{F}}}
\nc{\cG}{{\mathcal{G}}}
\nc{\cH}{{\mathcal{H}}}
\nc{\cI}{{\mathcal{I}}}
\nc{\cJ}{{\mathcal{J}}}
\nc{\cK}{{\mathcal{K}}}
\nc{\cL}{{\mathcal{L}}}
\nc{\cM}{{\mathcal{M}}}
\nc{\cN}{{\mathcal{N}}}
\nc{\cO}{{\mathcal{O}}}
\nc{\cP}{{\mathcal{P}}}
\nc{\cQ}{{\mathcal{Q}}}
\nc{\cR}{{\mathcal{R}}}
\nc{\cS}{{\mathcal{S}}}
\nc{\cT}{{\mathcal{T}}}
\nc{\cU}{{\mathcal{U}}}
\nc{\cV}{{\mathcal{V}}}
\nc{\cW}{{\mathcal{W}}}
\nc{\cX}{{\mathcal{X}}}
\nc{\cY}{{\mathcal{Y}}}
\nc{\cZ}{{\mathcal{Z}}}

\nc{\rc}{{\mathrm{c}}}
\nc{\rf}{{\mathrm{f}}}
\nc{\rh}{{\mathrm{h}}}
\nc{\rch}{{\mathrm{ch}}}
\nc{\rtd}{{\mathrm{td}}}

\nc{\rA}{{\mathrm{A}}}
\nc{\rB}{{\mathrm{B}}}
\nc{\rC}{{\mathrm{C}}}
\nc{\rD}{{\mathrm{D}}}
\nc{\rE}{{\mathrm{E}}}
\nc{\rF}{{\mathrm{F}}}
\nc{\rG}{{\mathrm{G}}}
\nc{\rH}{{\mathrm{H}}}
\nc{\rI}{{\mathrm{I}}}
\nc{\rJ}{{\mathrm{J}}}
\nc{\rK}{{\mathrm{K}}}
\nc{\rL}{{\mathrm{L}}}
\nc{\rM}{{\mathrm{M}}}
\nc{\rN}{{\mathrm{N}}}
\nc{\rO}{{\mathrm{O}}}
\nc{\rP}{{\mathrm{P}}}
\nc{\rQ}{{\mathrm{Q}}}
\nc{\rR}{{\mathrm{R}}}
\nc{\rS}{{\mathrm{S}}}
\nc{\rT}{{\mathrm{T}}}
\nc{\rU}{{\mathrm{U}}}
\nc{\rV}{{\mathrm{V}}}
\nc{\rW}{{\mathrm{W}}}
\nc{\rX}{{\mathrm{X}}}
\nc{\rY}{{\mathrm{Y}}}
\nc{\rZ}{{\mathrm{Z}}}

\nc{\bA}{{\mathbf{A}}}
\nc{\bB}{{\mathbf{B}}}
\nc{\bC}{{\mathbf{C}}}
\nc{\bD}{{\mathbf{D}}}
\nc{\bE}{{\mathbf{E}}}
\nc{\bF}{{\mathbf{F}}}
\nc{\bG}{{\mathbf{G}}}
\nc{\bH}{{\mathbf{H}}}
\nc{\bI}{{\mathbf{I}}}
\nc{\bJ}{{\mathbf{J}}}
\nc{\bK}{{\mathbf{K}}}
\nc{\bL}{{\mathbf{L}}}
\nc{\bM}{{\mathbf{M}}}
\nc{\bN}{{\mathbf{N}}}
\nc{\bO}{{\mathbf{O}}}
\nc{\bP}{{\mathbf{P}}}
\nc{\bQ}{{\mathbf{Q}}}
\nc{\bR}{{\mathbf{R}}}
\nc{\bS}{{\mathbf{S}}}
\nc{\bT}{{\mathbf{T}}}
\nc{\bU}{{\mathbf{U}}}
\nc{\bV}{{\mathbf{V}}}
\nc{\bW}{{\mathbf{W}}}
\nc{\bX}{{\mathbf{X}}}
\nc{\bY}{{\mathbf{Y}}}
\nc{\bZ}{{\mathbf{Z}}}

\nc{\ba}{{\mathbf{a}}}
\nc{\bb}{{\mathbf{b}}}
\nc{\bc}{{\mathbf{c}}}
\nc{\bd}{{\mathbf{d}}}
\nc{\be}{{\mathbf{e}}}
\nc{\bg}{{\mathbf{g}}}
\nc{\bh}{{\mathbf{h}}}
\nc{\bi}{{\mathbf{i}}}
\nc{\bj}{{\mathbf{j}}}
\nc{\bk}{{\mathbf{k}}}
\nc{\bl}{{\mathbf{l}}}
\nc{\bm}{{\mathbf{m}}}
\nc{\bn}{{\mathbf{n}}}
\nc{\bo}{{\mathbf{o}}}
\nc{\bp}{{\mathbf{p}}}
\nc{\bq}{{\mathbf{q}}}
\nc{\br}{{\mathbf{r}}}
\nc{\bs}{{\mathbf{s}}}
\nc{\bt}{{\mathbf{t}}}
\nc{\bu}{{\mathbf{u}}}
\nc{\bv}{{\mathbf{v}}}
\nc{\bw}{{\mathbf{w}}}
\nc{\bx}{{\mathbf{x}}}
\nc{\by}{{\mathbf{y}}}
\nc{\bz}{{\mathbf{z}}}

\nc{\fA}{{\mathfrak{A}}}
\nc{\fB}{{\mathfrak{B}}}
\nc{\fC}{{\mathfrak{C}}}
\nc{\fD}{{\mathfrak{D}}}
\nc{\fE}{{\mathfrak{E}}}
\nc{\fF}{{\mathfrak{F}}}
\nc{\fG}{{\mathfrak{G}}}
\nc{\fH}{{\mathfrak{H}}}
\nc{\fI}{{\mathfrak{I}}}
\nc{\fJ}{{\mathfrak{J}}}
\nc{\fK}{{\mathfrak{K}}}
\nc{\fL}{{\mathfrak{L}}}
\nc{\fM}{{\mathfrak{M}}}
\nc{\fN}{{\mathfrak{N}}}
\nc{\fO}{{\mathfrak{O}}}
\nc{\fP}{{\mathfrak{P}}}
\nc{\fQ}{{\mathfrak{Q}}}
\nc{\fR}{{\mathfrak{R}}}
\nc{\fS}{{\mathfrak{S}}}
\nc{\fT}{{\mathfrak{T}}}
\nc{\fU}{{\mathfrak{U}}}
\nc{\fV}{{\mathfrak{V}}}
\nc{\fW}{{\mathfrak{W}}}
\nc{\fX}{{\mathfrak{X}}}
\nc{\fY}{{\mathfrak{Y}}}
\nc{\fZ}{{\mathfrak{Z}}}

\nc{\fa}{{\mathfrak{a}}}
\nc{\fb}{{\mathfrak{b}}}
\nc{\fc}{{\mathfrak{c}}}
\nc{\fd}{{\mathfrak{d}}}
\nc{\fe}{{\mathfrak{e}}}
\nc{\ff}{{\mathfrak{f}}}
\nc{\fg}{{\mathfrak{g}}}
\nc{\fh}{{\mathfrak{h}}}
\nc{\fj}{{\mathfrak{j}}}
\nc{\fk}{{\mathfrak{k}}}
\nc{\fl}{{\mathfrak{l}}}
\nc{\fm}{{\mathfrak{m}}}
\nc{\fn}{{\mathfrak{n}}}
\nc{\fo}{{\mathfrak{o}}}
\nc{\fp}{{\mathfrak{p}}}
\nc{\fq}{{\mathfrak{q}}}
\nc{\fr}{{\mathfrak{r}}}
\nc{\fs}{{\mathfrak{s}}}
\nc{\ft}{{\mathfrak{t}}}
\nc{\fu}{{\mathfrak{u}}}
\nc{\fv}{{\mathfrak{v}}}
\nc{\fw}{{\mathfrak{w}}}
\nc{\fx}{{\mathfrak{x}}}
\nc{\fy}{{\mathfrak{y}}}
\nc{\fz}{{\mathfrak{z}}}

\nc{\sA}{{\mathsf{A}}}
\nc{\sB}{{\mathsf{B}}}
\nc{\sC}{{\mathsf{C}}}
\nc{\sD}{{\mathsf{D}}}
\nc{\sE}{{\mathsf{E}}}
\nc{\sF}{{\mathsf{F}}}
\nc{\sG}{{\mathsf{G}}}
\nc{\sH}{{\mathsf{H}}}
\nc{\sI}{{\mathsf{I}}}
\nc{\sJ}{{\mathsf{J}}}
\nc{\sK}{{\mathsf{K}}}
\nc{\sL}{{\mathsf{L}}}
\nc{\sM}{{\mathsf{M}}}
\nc{\sN}{{\mathsf{N}}}
\nc{\sO}{{\mathsf{O}}}
\nc{\sP}{{\mathsf{P}}}
\nc{\sQ}{{\mathsf{Q}}}
\nc{\sR}{{\mathsf{R}}}
\nc{\sS}{{\mathsf{S}}}
\nc{\sT}{{\mathsf{T}}}
\nc{\sU}{{\mathsf{U}}}
\nc{\sV}{{\mathsf{V}}}
\nc{\sW}{{\mathsf{W}}}
\nc{\sX}{{\mathsf{X}}}
\nc{\sY}{{\mathsf{Y}}}
\nc{\sZ}{{\mathsf{Z}}}

\nc{\sa}{{\mathsf{a}}}
\nc{\sd}{{\mathsf{d}}}
\nc{\se}{{\mathsf{e}}}
\nc{\sg}{{\mathsf{g}}}
\nc{\sh}{{\mathsf{h}}}
\nc{\si}{{\mathsf{i}}}
\nc{\sj}{{\mathsf{j}}}
\nc{\sk}{{\mathsf{k}}}
\nc{\sn}{{\mathsf{n}}}
\nc{\so}{{\mathsf{o}}}
\nc{\sq}{{\mathsf{q}}}
\nc{\sr}{{\mathsf{r}}}
\nc{\st}{{\mathsf{t}}}
\nc{\su}{{\mathsf{u}}}
\nc{\sv}{{\mathsf{v}}}
\nc{\sw}{{\mathsf{w}}}
\nc{\sx}{{\mathsf{x}}}
\nc{\sy}{{\mathsf{y}}}
\nc{\sz}{{\mathsf{z}}}

\nc{\oA}{{\overline{A}}}
\nc{\oB}{{\overline{B}}}
\nc{\oC}{{\overline{C}}}
\nc{\oD}{{\overline{D}}}
\nc{\oE}{{\overline{E}}}
\nc{\oF}{{\overline{F}}}
\nc{\oG}{{\overline{G}}}
\nc{\oH}{{\overline{H}}}
\nc{\oI}{{\overline{I}}}
\nc{\oJ}{{\overline{J}}}
\nc{\oK}{{\overline{K}}}
\nc{\oL}{{\overline{L}}}
\nc{\oM}{{\overline{M}}}
\nc{\oN}{{\overline{N}}}
\nc{\oO}{{\overline{O}}}
\nc{\oP}{{\overline{P}}}
\nc{\oQ}{{\overline{Q}}}
\nc{\oR}{{\overline{R}}}
\nc{\oS}{{\overline{S}}}
\nc{\oT}{{\overline{T}}}
\nc{\oU}{{\overline{U}}}
\nc{\oV}{{\overline{V}}}
\nc{\oW}{{\overline{W}}}
\nc{\oX}{{\overline{X}}}
\nc{\oY}{{\overline{Y}}}
\nc{\oZ}{{\overline{Z}}}

\nc{\oa}{{\overline{a}}}
\nc{\ob}{{\overline{b}}}
\nc{\oc}{{\overline{c}}}
\nc{\od}{{\overline{d}}}
\nc{\of}{{\overline{f}}}
\nc{\og}{{\overline{g}}}
\nc{\oh}{{\overline{h}}}
\nc{\oi}{{\overline{i}}}
\nc{\oj}{{\overline{j}}}
\nc{\ok}{{\overline{k}}}
\nc{\ol}{{\overline{l}}}
\nc{\om}{{\overline{m}}}
\nc{\on}{{\overline{n}}}
\nc{\oo}{{\overline{o}}}
\nc{\op}{{\overline{p}}}
\nc{\oq}{{\overline{q}}}
\nc{\os}{{\overline{s}}}
\nc{\ot}{{\overline{t}}}
\nc{\ou}{{\overline{u}}}
\nc{\ov}{{\overline{v}}}
\nc{\ow}{{\overline{w}}}
\nc{\ox}{{\overline{x}}}
\nc{\oy}{{\overline{y}}}
\nc{\oz}{{\overline{z}}}

\nc{\tA}{{\tilde{A}}}
\nc{\tB}{{\tilde{B}}}
\nc{\tC}{{\tilde{C}}}
\nc{\tD}{{\tilde{D}}}
\nc{\tE}{{\tilde{E}}}
\nc{\tF}{{\tilde{F}}}
\nc{\tG}{{\tilde{G}}}
\nc{\tH}{{\tilde{H}}}
\nc{\tI}{{\tilde{I}}}
\nc{\tJ}{{\tilde{J}}}
\nc{\tK}{{\tilde{K}}}
\nc{\tL}{{\tilde{L}}}
\nc{\tM}{{\tilde{M}}}
\nc{\tN}{{\tilde{N}}}
\nc{\tO}{{\tilde{O}}}
\nc{\tP}{{\tilde{P}}}
\nc{\tQ}{{\tilde{Q}}}
\nc{\tR}{{\tilde{R}}}
\nc{\tS}{{\tilde{S}}}
\nc{\tT}{{\tilde{T}}}
\nc{\tU}{{\tilde{U}}}
\nc{\tV}{{\tilde{V}}}
\nc{\tW}{{\tilde{W}}}
\nc{\tX}{{\tilde{X}}}
\nc{\tY}{{\tilde{Y}}}
\nc{\tZ}{{\tilde{Z}}}

\nc{\tfD}{{\tilde{\fD}}}
\nc{\tcA}{{\tilde{\cA}}}
\nc{\tcC}{{\tilde{\cC}}}
\nc{\tcD}{{\tilde{\cD}}}
\nc{\tcE}{{\tilde{\cE}}}
\nc{\tcF}{{\tilde{\cF}}}
\nc{\tcM}{{\tilde{\cM}}}
\nc{\tcT}{{\tilde{\cT}}}

\nc{\ta}{{\tilde{a}}}
\nc{\tb}{{\tilde{b}}}
\nc{\tc}{{\tilde{c}}}
\nc{\td}{{\tilde{d}}}
\nc{\te}{{\tilde{e}}}
\nc{\tf}{{\tilde{f}}}
\nc{\tg}{{\tilde{g}}}
\nc{\ti}{{\tilde{\imath}}}
\nc{\tj}{{\tilde{j}}}
\nc{\tk}{{\tilde{k}}}
\nc{\tl}{{\tilde{l}}}
\nc{\tm}{{\tilde{m}}}
\nc{\tn}{{\tilde{n}}}
\nc{\tp}{{\tilde{p}}}
\nc{\tq}{{\tilde{q}}}
\nc{\ts}{{\tilde{s}}}
\nc{\tu}{{\tilde{u}}}
\nc{\tv}{{\tilde{v}}}
\nc{\tw}{{\tilde{w}}}
\nc{\tx}{{\tilde{x}}}
\nc{\ty}{{\tilde{y}}}
\nc{\tz}{{\tilde{z}}}

\nc{\hA}{{\hat{A}}}
\nc{\hB}{{\hat{B}}}
\nc{\hC}{{\hat{C}}}
\nc{\hD}{{\hat{D}}}
\nc{\hE}{{\hat{E}}}
\nc{\hF}{{\hat{F}}}
\nc{\hG}{{\hat{G}}}
\nc{\hH}{{\hat{H}}}
\nc{\hI}{{\hat{I}}}
\nc{\hJ}{{\hat{J}}}
\nc{\hK}{{\hat{K}}}
\nc{\hL}{{\hat{L}}}
\nc{\hM}{{\hat{M}}}
\nc{\hN}{{\hat{N}}}
\nc{\hO}{{\hat{O}}}
\nc{\hP}{{\hat{P}}}
\nc{\hQ}{{\hat{Q}}}
\nc{\hR}{{\hat{R}}}
\nc{\hS}{{\widehat{S}}}
\nc{\hT}{{\hat{T}}}
\nc{\hU}{{\widehat{U}}}
\nc{\hV}{{\hat{V}}}
\nc{\hW}{{\hat{W}}}
\nc{\hX}{{\widehat{X}}}
\nc{\hY}{{\widehat{Y}}}
\nc{\hZ}{{\hat{Z}}}

\nc{\hcB}{{\widehat{\cB}}}

\nc{\ha}{{\hat{a}}}
\nc{\hb}{{\hat{b}}}
\nc{\hc}{{\hat{c}}}
\nc{\hd}{{\hat{d}}}
\nc{\he}{{\hat{e}}}
\nc{\hg}{{\hat{g}}}
\nc{\hh}{{\hat{h}}}
\nc{\hi}{{\hat{i}}}
\nc{\hj}{{\hat{j}}}
\nc{\hk}{{\hat{k}}}
\nc{\hl}{{\hat{l}}}
\nc{\hm}{{\hat{m}}}
\nc{\hn}{{\hat{n}}}
\nc{\ho}{{\hat{o}}}
\nc{\hp}{{\hat{p}}}
\nc{\hq}{{\hat{q}}}
\nc{\hr}{{\hat{r}}}
\nc{\hs}{{\hat{s}}}
\nc{\hu}{{\hat{u}}}
\nc{\hv}{{\hat{v}}}
\nc{\hw}{{\hat{w}}}
\nc{\hx}{{\hat{x}}}
\nc{\hy}{{\hat{y}}}
\nc{\hz}{{\hat{z}}}

\nc{\hcC}{{\widehat{\cC}}}
\nc{\hcT}{{\widehat{\cT}}}

\nc{\eps}{\varepsilon}
\nc{\lan}{\big\langle}
\nc{\ran}{\big\rangle}
\nc{\kk}{{\Bbbk}}
\nc{\io}{\upiota}
\nc{\Kr}{\mathsf{Kr}}
\nc{\cKr}{\mathcal{K}\!\mathit{r}}

\nc{\perf}{\mathrm{perf}}
\nc{\Dm}{\bD^{-}}
\nc{\Db}{\bD^{\mathrm{b}}}
\nc{\Dp}{\bD^{\mathrm{perf}}}
\nc{\Dperf}{\bD^{\mathrm{perf}}}
\nc{\Dqc}{\bD_{\mathrm{qc}}}
\nc{\Du}{\bD}
\nc{\Dsing}{\bD^{\mathrm{sg}}}
\nc{\Dg}{\bD^{\mathrm{sg}}}

\DeclareMathOperator{\tr}{\mathbf{tr}}

\nc{\Rn}{\rR_{\mathrm{node}}}
\nc{\Cn}{\cC_{\mathrm{node}}}

\def\bw#1#2{\textstyle{\bigwedge\hskip-0.9mm^{#1}}\hskip0.2mm{#2}}

\DeclareMathOperator{\Hom}{\mathrm{Hom}}

\DeclareMathOperator{\Ext}{\mathrm{Ext}}

\DeclareMathOperator{\cEnd}{\mathcal{E}\!\mathit{nd}}

\DeclareMathOperator{\RHom}{\mathrm{\mathbf{R}Hom}}
\DeclareMathOperator{\cRHom}{\mathbf{R}\mathcal{H}\mathit{om}}

\DeclareMathOperator{\Spec}{\mathrm{Spec}}

\DeclareMathOperator{\Coh}{\mathrm{Coh}}

\DeclareMathOperator{\Bl}{\mathrm{Bl}}

\DeclareMathOperator{\Pic}{\mathrm{Pic}}

\DeclareMathOperator{\Cliff}{\mathcal{C}\!\ell}

\DeclareMathOperator{\Sym}{\mathrm{Sym}}
\DeclareMathOperator{\Ker}{\mathrm{Ker}}
\DeclareMathOperator{\Coker}{\mathrm{Coker}}

\DeclareMathOperator{\Cone}{\mathrm{Cone}}

\DeclareMathOperator{\pr}{\mathrm{pr}}
\DeclareMathOperator{\Qu}{\mathsf{Qu}}

\DeclareMathOperator{\Gr}{\mathrm{Gr}}

\DeclareMathOperator{\g}{\mathrm{g}}
\DeclareMathOperator{\id}{\mathrm{id}}
\DeclareMathOperator{\rank}{\mathrm{rk}}

\DeclareMathOperator{\Br}{Br}

\def\Pinfty#1{\P^{\infty,{#1}}}

\newcommand{\typeo}[1]{$\mathbf{#1}$}
\newcommand{\typemm}[2]{\mbox{$({#1}\textrm{-}{#2})$}}

\theoremstyle{plain}

\newtheorem{theorem}{Theorem}[section]
\newtheorem{conjecture}[theorem]{Conjecture}

\newtheorem{lemma}[theorem]{Lemma}
\newtheorem{proposition}[theorem]{Proposition}
\newtheorem{corollary}[theorem]{Corollary}

\theoremstyle{definition}

\newtheorem{definition}[theorem]{Definition}

\newtheorem{example}[theorem]{Example}

\theoremstyle{remark}

\newtheorem{remark}[theorem]{Remark}

\newenvironment{renumerate}{\begin{enumerate}[label={\textup{(\roman*)}}]}{\end{enumerate}}
\newenvironment{alenumerate}{\begin{enumerate}[label={\textup{(\alph*)}}]}{\end{enumerate}}

\title[Spinor modifications of conic bundles and derived categories of Fano threefolds]%
{Spinor modifications of conic bundles and\\derived categories of 1-nodal Fano threefolds}
\author{Alexander Kuznetsov}
\address{{\sloppy
\parbox{0.9\textwidth}{
Algebraic Geometry Section, Steklov Mathematical Institute of Russian Academy of Sciences,\\
8 Gubkin str., Moscow 119991 Russia
\\[5pt]
Laboratory of Algebraic Geometry, National Research University Higher School of Economics, Russian Federation
}\bigskip}}
\email{akuznet@mi-ras.ru}
\date{}

\thanks{I was partially supported by the HSE University Basic Research Program.}

\begin{document}

\maketitle

\begin{abstract}
Given a flat conic bundle~$X/S$ and an \emph{abstract spinor bundle}~$\cF$ on~$X$ 
we define a new conic bundle~$X_\cF/S$, 
called a \emph{spinor modification} of~$X$, 
such that the even Clifford algebras of~$X/S$ and~$X_\cF/S$ are Morita equivalent
and the orthogonal complements of~$\Db(S)$ in~$\Db(X)$ and~$\Db(X_\cF)$ are equivalent as well.
We demonstrate how the technique of spinor modifications works
in the example of conic bundles associated with some nonfactorial 1-nodal prime Fano threefolds.
In particular, we construct a categorical absorption of singularities for these Fano threefolds.
\end{abstract}


\section{Introduction}

One of the most important types of Mori fiber spaces in the minimal model program are \emph{conic bundles},
i.e., projective dominant morphisms~$f \colon X \to S$ of relative dimension~$1$
with relatively ample anticanonical class.
Conic bundles are extremely useful for the study of birational geometry of~$X$.
For instance, in dimension~3, the intermediate Jacobian of~$X$ 
can be computed from the discriminant divisor~$\Delta_{X/S}$ of~$f$
and there are powerful non-rationality criteria for~$X$ in terms of~$\Delta_{X/S}$,
see~\cite{P18} for a recent survey.

On the other hand, if a conic bundle is flat 
(this is always the case when the total space of a conic bundle 
is a Gorenstein threefold, see~\cite[Theorem~7]{Cutkosky}),
one can also control the bounded derived category~$\Db(X)$ of~$X$;
it has a semiorthogonal decomposition with components~$\Db(S)$ and~$\Db(S, \Cliff_0(X/S))$,
where~$\Cliff_0(X/S)$ is the \emph{even Clifford algebra} of~$X/S$.
While the first component of this decomposition is quite familiar 
(e.g., if~$X$ is a smooth rationally connected threefold, $S$ is a smooth rational surface),
the second component is not so easy to understand, especially when~$\Cliff_0(X/S)$ has a complicated structure.
The goal of this paper is to develop tools for understanding this category
and demonstrate how these tools work in the geometrically interesting examples 
of conic bundles corresponding to nonfactorial 1-nodal prime Fano threefolds.

So, from now on we assume that~$f \colon X \to S$ is a flat conic bundle 
(of any dimension, not necessarily relatively minimal;
see Definition~\ref{def:cb} for the assumptions we impose).
Let~$q \colon \cL \to \Sym^2\cE^\vee$ be the corresponding quadratic form, 
where~$\cL$ is a line bundle and~$\cE$ is a vector bundle of rank~3 on~$S$,
so that~\mbox{$X \subset \P_S(\cE)$} is a divisor of relative degree~2 with equation given by~$q$.
Then
\begin{equation}
\label{eq:cl0}
\Cliff_0(X/S) = \Cliff_0(q) = \Cliff_0(\cE,\cL,q) \coloneqq \cO_S \oplus (\wedge^2\cE \otimes \cL),
\end{equation}
and the multiplication in~$\Cliff_0(q)$ is induced by wedge product and the form~$q$, 
see~\S\ref{ss:pca} and Example~\ref{ex:clifford} for explicit formulas.
The {\sf kernel category} of~$X$ is the subcategory in~$\Db(X)$ defined by
\begin{equation}
\label{eq:def-kernel}
\Ker(X/S) = \Ker(f_*) \coloneqq \{ \cG \in \Db(X) \mid f_*(\cG) = 0 \} \subset \Db(X).
\end{equation}
By~\cite[Theorem~4.2]{K08} there is a semiorthogonal decomposition~$\Db(X) = \langle \Ker(X/S), f^*(\Db(S)) \rangle$
and an equivalence of triangulated categories
\begin{equation}
\label{eq:ker-f}
\Ker(X/S) \simeq \Db(S, \Cliff_0(q)),
\end{equation}
where the right-hand side is the bounded derived category of coherent sheaves of modules over~$\Cliff_0(q)$.
As we explained above, our goal is to study this category.

\subsection{Spinor modifications}

The starting point of our approach to the study of~$\Ker(X/S)$ is a simple observation:
there are, in fact, many different flat conic bundles over~$S$ that have the same kernel category,
but different (though Morita equivalent) even Clifford algebras, 
and some of these Clifford algebras are easier to deal with.
The goal is, therefore, to take control over such conic bundles.

One construction that allows one to pass between conic bundles (or, more generally, quadric bundles) 
with equivalent kernel categories was developed in~\cite{K24}.
This is an iteration of two operations: the \emph{hyperbolic reduction}, 
that decreases the relative dimension of a quadric bundle,
and the \emph{hyperbolic extension}, that increases the relative dimension;
together, they generate an equivalence relation, called \emph{hyperbolic equivalence}.
Thus, to pass from one conic bundle to another, 
one needs to consider intermediate quadric bundles of higher dimension,
which, of course, is a disadvantage of this approach.

In this paper we suggest another way to find conic bundles with equivalent kernel category 
(and Morita equivalent Clifford algebra).
For this we introduce the following notion.

\begin{definition}
\label{def:abs}
A vector bundle~$\cF$ of rank~$2$ on a conic bundle~$f \colon X \to S$ is {\sf an abstract spinor bundle} if
\begin{enumerate}[label={\textup{(\roman*)}}]
\item 
\label{def:asb-ker}
$f_*(\cF) = 0$, and
\item 
\label{def:asb-c1}
$\rc_1(\cF) = K_{X/S}$ in~$\Pic(X/S)$,
\end{enumerate}
where~$\Pic(X/S) \coloneqq \Pic(X) / f^*\Pic(S)$ is the relative Picard group of~$X/S$.
\end{definition}

By~\cite{K08} a conic bundle~$X/S$ has a natural sequence of \emph{canonical spinor bundles}~$\cF^i_{X/S}$,
the images of the standard modules~$\Cliff_i(q)$ over~$\Cliff_0(q)$ under the equivalence~\eqref{eq:ker-f}
(see also~\eqref{eq:cf-io} or, for a down-to-earth description, Lemma~\ref{lem:cf0-cf1}),
but they do not exhaust all possibilities.
In fact, the main result of this paper is the following \emph{modification theorem}.

We denote by~$\cEnd^0(\cF) \subset \cEnd(\cF)$ the trace-free part of the endomorphism bundle of~$\cF$.

\begin{theorem}
\label{thm:asb-cb}
Let~$X/S$ be a flat conic bundle with quadratic form~$q$.
For any abstract spinor bundle~$\cF$ on~$X$ there is a flat conic bundle~$X_\cF \subset \P_S(\cE_\cF)$ over~$S$
with quadratic form~$q_\cF \colon \cL_\cF \to \Sym^2\cE_\cF^\vee$, where
\begin{equation*}
\cL_\cF \cong \det(f_*\cEnd^0(\cF)) 
\qquad\text{and}\qquad 
\cE_\cF \cong (f_*\cEnd^0(\cF))^\vee,
\end{equation*}
such that~$\Cliff_0(q_\cF) \cong f_*\cEnd(\cF)$ and
\begin{renumerate}
\item 
\label{it:cliff-morita}
there is an $S$-linear Morita equivalence of algebras~$\Cliff_0(q_\cF) \sim \Cliff_0(q)$, and
\item 
\label{it:ker-derived}
there is an $S$-linear t-exact Fourier--Mukai equivalence of categories~$\Ker(X_\cF/S) \simeq \Ker(X/S)$.
\end{renumerate}
Moreover, the equivalence in~\ref{it:ker-derived} 
takes the canonical spinor bundle~$\cF^0_{X_\cF/S} \in \Ker(X_\cF/S)$ to~$\cF \in \Ker(X/S)$.

Conversely, if~$X'/S$ is a flat conic bundle with quadratic form~$q'$ such that there is an equivalence
\begin{equation*}
\Cliff_0(q') \sim \Cliff_0(q)
\qquad\text{or}\qquad 
\Ker(X'/S) \simeq \Ker(X/S),
\end{equation*}
as in~\ref{it:cliff-morita} or~\ref{it:ker-derived},
then there is an abstract spinor bundle~$\cF$ on~$X$ such that~$X' \cong X_\cF$.
\end{theorem}

The conic bundle~$X_\cF/S$ produced from~$X/S$ and the abstract spinor bundle~$\cF$ on~$X$
will be called the {\sf $\cF$-modification} of~$X/S$, 
or, if we do not want to specify~$\cF$, simply {\sf a spinor modification} of~$X/S$.

By Theorem~\ref{thm:asb-cb} classification of conic bundles~$X'/S$ with a Morita equivalence~$\Cliff_0(q') \sim \Cliff_0(q)$
reduces to classification of abstract spinor bundles on~$X$.

As we mentioned above, hyperbolic equivalent conic bundles have Morita equivalent Clifford algebras.
In a combination with the converse part of Theorem~\ref{thm:asb-cb} this has the following simple consequence 

\begin{corollary}
\label{cor:he-sm}
Any conic bundle hyperbolic equivalent to~$X/S$ is a spinor modification of~$X/S$.
\end{corollary}

We expect that the converse is also true:

\begin{conjecture}
\label{conj:asb-he}
If~$X'/S$ is a spinor modification of~$X/S$ then~$X'/S$ is hyperbolic equivalent to~$X/S$.
\end{conjecture}

\subsection{Outline of the proof}

We prove Theorem~\ref{thm:asb-cb} in~\S\ref{sec:cb}. The proof consists of three steps.

First, we show that the isomorphism class of the restriction of an abstract spinor bundle 
to a geometric fiber of a conic bundle only depends on the isomorphism class of the fiber, 
see Proposition~\ref{prop:asb-fibers}.
In particular, it follows that any abstract spinor bundle~$\cF$ 
is fiberwise isomorphic to the canonical spinor bundle~$\cF^0_{X/S}$, 
and therefore the algebra~$f_*\cEnd(\cF)$ on~$S$ is a \emph{pointwise Clifford algebra}, see Definition~\ref{def:pca}.

Next, we observe that the even Clifford algebra~$\Cliff_0(q)$ of a quadratic form~$q$ 
contains all the information about~$q$, see Example~\ref{ex:clifford},
and using this observation we check that any pointwise Clifford algebra
is isomorphic to the even Clifford algebra of an appropriate quadratic form, 
see Proposition~\ref{prop:pca-cliff}.

\begin{remark}
The result proved in Proposition~\ref{prop:pca-cliff} 
is similar to~\cite[Theorem~6.12]{CI},
where a bijection between flat conic bundles (over a smooth base scheme)
and locally Clifford algebras is established.
Note that~$\cR$ is a \emph{locally Clifford algebra} (\cite[Definition~4.6]{CI}) on a scheme~$S$
if for any closed point~$s$ the localization of~$\cR$ at~$s$ 
is the Clifford algebra of a family of quadratic forms,
and~$\cR$ is a pointwise Clifford algebra if for any geometric point~$s$ the fiber of~$\cR$ at~$s$
is the even Clifford algebra of a quadratic form.
The latter notion is obviously weaker, hence the bijection proved in Proposition~\ref{prop:pca-cliff} is stronger.

Also, the proof of Proposition~\ref{prop:pca-cliff} is more direct and transparent 
than the proof of~\cite[Theorem~6.12]{CI};
instead of relying on the notion of a Severi--Brauer scheme of an arbitrary family of algebras
it reconstructs the quadratic form of the conic bundle 
directly from the commutator map of a pointwise Clifford algebra.
\end{remark}

To deduce Theorem~\ref{thm:asb-cb} from Propositions~\ref{prop:pca-cliff} and~\ref{prop:asb-fibers}
we check that any abstract spinor bundle~$\cF$
is a relative tilting bundle for the category~$\Ker(X/S)$ over~$S$.
Therefore, on the one hand, there is an equivalence~$\Ker(X/S) \simeq \Db(S, f_*\cEnd(\cF))$,
and, on the other hand, $f_*\cEnd(\cF)$ is a pointwise Clifford algebra, 
so that~$f_*\cEnd(\cF) \cong \Cliff_0(q_\cF)$ for an appropriate quadratic form~$q_\cF$.
Thus,
\begin{equation*}
\Ker(X/S) \simeq \Db(S, \Cliff_0(q_\cF)) \simeq \Ker(X_\cF/S),
\end{equation*}
where~$X_\cF$ is the conic bundle associated with~$q_\cF$
and the second equivalence is~\eqref{eq:ker-f} for~$X_\cF$.

One thing we want to point out is that the combination of the above steps
makes the construction of the spinor modification~$X_\cF$ completely explicit and effective;
we demonstrate this in~\S\ref{sec:af3-cb}.

To prove the converse part of Theorem~\ref{thm:asb-cb} 
we show that the image of the canonical spinor bundle~$\cF^0_{X'/S}$
under any $S$-linear t-exact equivalence~$\Ker(X'/S) \simeq \Ker(X/S)$ 
is an abstract spinor bundle~$\cF$ on~$X$.
Then it is easy to see that~$X_\cF \cong X'$.

We finish~\S\ref{sec:cb} by discussing a few properties of spinor modifications.
First, we check that the spinor modification relation on the set of all flat conic bundles 
is an equivalence relation, see Corollary~\ref{cor:spinor-modification}.
We also show that any spinor modification~$X_\cF/S$ of a conic bundle~$X/S$
is birational to~$X$ over~$S$, see Lemma~\ref{lem:sm-bir}
and that~$X_\cF$ is regular or smooth if and only if~$X$ is regular or smooth, see Proposition~\ref{prop:xprime-smooth}.

\subsection{Application to 1-nodal Fano threefolds}

In~\S\ref{sec:af3-cb} we show how Theorem~\ref{thm:asb-cb} works 
in an example of geometric interest:
for conic bundles~$Y \to \P^2$ providing small resolutions 
of nonfactorial 1-nodal Fano threefolds with Picard number~1
classified in~\cite{KP23}.
There are~4 interesting conic bundles of this sort
(altogether there are~6 types, but two of them are projectivizations of vector bundles, 
so for them the Clifford algebra is Morita-trivial):
in the notation of~\cite[Table~2]{KP23} 
they correspond to 1-nodal Fano threefolds~$X$ of types
\begin{equation*}
\text{\typeo{12nb}, \typeo{10na}, \typeo{8nb}, and~\typeo{5n}},
\end{equation*}
of genus~12, 10, 8, and~5, respectively.
So, we consider small resolutions~$\pi \colon Y \to X$ of such threefolds
that have a structure of a conic bundle~$Y \to \P^2$.
The quadratic forms associated with these conic bundles can be described explicitly,
see~\eqref{eq:k} and~\eqref{eq:cev} for the first three and~\eqref{eq:cev-5n} for the last.

We introduce a uniform construction of an interesting abstract spinor bundle~$\cF$ in all these cases,
applying Serre's construction to the (unique) $K$-trivial curve in~$Y$
(i.e., the exceptional curve of the small contraction~$\pi \colon Y \to X$), see Lemma~\ref{lem:pf-cf1}.
In the first three cases we show that~$\cF$ is exceptional (see Corollary~\ref{cor:cf-exc})
and identify the spinor modifications~$Y_\cF$ with
\begin{itemize}
\item 
a divisor~$Y_\cF \subset \P^2 \times \P^2$ of bidegree~$(2,1)$, for type~\typeo{12nb},
\item 
a double covering~$Y_\cF \to \P^1 \times \P^2$ branched at a divisor of bidegree~$(2,2)$, for type~\typeo{10na}, and 
\item 
the blowup~$Y_\cF \to \bar{Y}$ of a smooth cubic threefold~$\bar{Y} \subset \P^4$ in a line, for type~\typeo{8nb},
\end{itemize}
see Proposition~\ref{prop:good-cb0} and Corollary~\ref{cor:yprime-12-10-8}.
Using these identifications and the equivalences
\begin{equation*}
\Ker(Y/\P^2) \simeq \Ker(Y_\cF/\P^2)
\end{equation*}
proved in Theorem~\ref{thm:asb-cb},
we identify the orthogonal complement~$\cF^\perp \subset \Ker(Y/\P^2)$ of~$\cF$ with
\begin{itemize}
\item 
the derived category of a quiver with two vertices and three arrows, see Proposition~\ref{prop:cby-12},
\item 
the derived category of a curve of genus~2, see Proposition~\ref{prop:cby-10}, and
\item 
a component of the derived category of the cubic threefold~$\bar{Y}$, see Proposition~\ref{prop:cby-8},
\end{itemize}
for types~\typeo{12nb}, \typeo{10na}, and~\typeo{8nb}, respectively.

We use these results to describe the derived categories~$\Db(X)$ 
of the corresponding 1-nodal Fano threefolds of types~\typeo{12nb}, \typeo{10na}, and~\typeo{8nb}.
We show that, up to twist, the spinor bundle~$\cF$ 
is isomorphic to the pullback of a vector bundle, denoted further by~$\cU_X$,
along the small contraction~$\pi \colon Y \to X$.
We also show that~$\cU_X$ is a \emph{Mukai bundle} on~$X$, see Proposition~\ref{prop:cf1},
and that there is a semiorthogonal decomposition
\begin{equation*}
\Db(X) = \langle \rP_X, \cA_X, \cU_X, \cO_X \rangle,
\end{equation*}
where~$\rP_X$ is a \emph{$\P^{\infty,2}$-object} (in the sense of~\cite{KS22}),
and~$\cA_X$ is a smooth and proper category equivalent to the category~$\cF^\perp$ described above,
see Theorem~\ref{thm:absorption}.   
In particular, $\rP_X$ provides a universal deformation absorption of singularities of~$X$,
and for any smoothing~$\cX/B$ of~$X$ over a pointed curve~$(B,o)$
there is a smooth and proper over~$B$ category~$\cA$ 
with central fiber~$\cA_o = \cA_X$ and with general fiber~$\cA_b \subset \Db(\cX_b)$ 
equivalent to the orthogonal complement of the structure sheaf and the Mukai bundle of the smooth Fano threefold~$\cX_b$,
see Corollary~\ref{cor:def-abs}.

For the conic bundle~$Y \to \P^2$ of the small resolution~$\pi \colon Y \to X$
of a nonfactorial 1-nodal Fano threefold~$X$ of type~\typeo{5n} the object~$\cF$ is not exceptional.
In this case we do not have a precise result similar to the one for types~\typeo{12nb}, \typeo{10na}, or~\typeo{8nb}.
However, we show that there is a semiorthogonal decomposition
\begin{equation*}
\Db(X) = \langle \rP_X, \cA_X, \cO_X \rangle,
\end{equation*}
where~$\rP_X$ is again a~$\P^{\infty,2}$-object
and~$\cA_X$ is a smooth and proper category which is a Krull--Schmidt partner of~$\Ker(Y/\P^2)$ in the sense of~\cite{O16-KS}.
We also show that~$\cA_X$ deforms to the orthogonal complement of~$\cO_{\cX_b}$ in~$\Db(\cX_b)$
for any smoothing~$\cX$ of~$X$.

\subsection*{Conventions}

All schemes are separated schemes of finite type over a field~$\kk$ of characteristic different from~2.
We denote by~$\Db(S)$ the bounded derived category of coherent sheaves on~$S$
and by~$\Dp(S)$ the category of perfect complexes.
Similarly, if a sheaf of~$\cO_S$-algebras~$\cR$ is given, 
we write~$\Db(S,\cR)$ for the bounded derived categories of coherent right $\cR$-modules.

All functors are derived by default; 
in particular, we write~$f_*$ and~$f^*$ for the derived pushforward and pullback functors,
$\otimes$ for the derived tensor product, and given an object~$\cF$, 
we write~$\cF_s$ and~$\cF\vert_Z$ for the derived restriction of~$\cF$ to a closed point~$s$ or a closed subscheme~$Z$, respectively.

\subsection*{Acknowledgements}

I would like to thank the anonymous referee for his comments about the paper, 
and especially for attracting my attention to~\cite{CI}.


\section{Spinor modifications of conic bundles}
\label{sec:cb}

For the purposes of this paper we adopt the following 

\begin{definition}
\label{def:cb}
A {\sf conic bundle}~$f \colon X \to S$ is a flat projective Gorenstein morphism 
of relative dimension~$1$ such that~$f_*\cO_X \cong \cO_S$, 
the relative anticanonical class~$-K_{X/S}$ is $f$-ample,
and~$S$ is integral.
\end{definition}

The fiber of a conic bundle~$f \colon X \to S$ over a geometric point~$s \in S$
is a connected projective Gorenstein curve~$X_s$ 
and its anticanonical line bundle~$\omega^{-1}_{X_s} \cong \omega^{-1}_{X/S}\vert_{X_s}$ is ample.
Thus, every geometric fiber of~$f$ is isomorphic to a plane conic
(with respect to the anticanonical embedding).
It also follows that
\begin{equation*}
\cE \coloneqq f_*(\omega_{X/S}^{-1})^\vee 
\end{equation*}
is a vector bundle of rank~3, and the natural morphism~$\Sym^2\cE^\vee = \Sym^2(f_*(\omega_{X/S}^{-1})) \to f_*(\omega_{X/S}^{-2})$
is an epimorphism onto a vector bundle of rank~5, hence its kernel is a line bundle, which we denote~$\cL$.
The induced embedding of vector bundles
\begin{equation*}
q \colon \cL \to \Sym^2\cE^\vee
\end{equation*}
can be thought of as a family of nowhere vanishing quadratic forms on~$\cE$.
It defines a global section of the line bundle~$p^*\cL^\vee \otimes \cO_{\P_S(\cE)/S}(2)$ on~$\P_S(\cE)$,
where~$p \colon \P_S(\cE) \to S$ is the natural projection,
whose zero locus is exactly~$X \subset \P_S(\cE)$.

Conversely, given a vector bundle~$\cE$ of rank~$3$, a line bundle~$\cL$, 
and a nowhere vanishing quadratic form~$q \colon \cL \to \Sym^2\cE^\vee$ on~$S$,
we can define the subscheme~$X \subset \P_S(\cE)$ as the zero locus 
of the induced global section of~$p^*\cL^\vee \otimes \cO_{\P_S(\cE)/S}(2)$ on~$\P_S(\cE)$;
then the morphism~$f \coloneqq p\vert_X \colon X \to S$ is a conic bundle.
Furthermore, if~$\cM$ is a line bundle on~$S$, the quadratic form~$q$ 
can be also considered as a quadratic form~$\cL \otimes \cM^{\otimes 2} \to \Sym^2(\cE \otimes \cM^\vee)^\vee$,
whose zero locus in~$\P_S(\cE \otimes \cM^\vee) = \P_S(\cE)$ coincides with~$X$.
Thus, we obtain a bijection between conic bundles and quadratic forms up to twists.

\begin{remark}
\label{rem:normalized-twist}
It is easy to check that starting with a quadratic form~$q \colon \cL \to \Sym^2\cE^\vee$ on~$S$
and using the above constructions first to produce from it a conic bundle, 
and then to produce a quadratic form again, 
we obtain the \emph{normalized twist} of the same quadratic form~$q \colon \cL' \to \Sym^2{\cE'}^\vee$, where
\begin{equation*}
\cL' \coloneqq \cL \otimes \det(\cE)^{\otimes 2} \otimes \cL^{\otimes 2} \cong \det(\cE)^{\otimes 2} \otimes \cL^{\otimes 3},
\qquad
\cE' \coloneqq \cE \otimes \det(\cE)^\vee \otimes \cL^\vee \cong (\wedge^2\cE \otimes \cL)^\vee.
\end{equation*}
Note that normalized quadratic forms are characterized by the existence of an isomorphism
\begin{equation}
\label{eq:normalization}
\cL \cong \det(\cE)^\vee.
\end{equation} 
In particular, any quadratic form has a unique normalized twist.
\end{remark}

\subsection{Pointwise Clifford algebras}
\label{ss:pca}

Given a quadratic form~$q \colon \cL \to \Sym^2\cE^\vee$ 
we define the even Clifford $\cO_S$-algebra~$\Cliff_0(q)$ by the formula~\eqref{eq:cl0} from the introduction,
and we endow it with multiplication as follows.
The first component~$\cO_S$ of~$\Cliff_0(q)$ is generated by the unit~$\mathbf{1}$ of the algebra,
and the multiplication on the second component~$\wedge^2\cE \otimes \cL$ is defined as the sum of the following maps
\begin{equation}
\label{eq:cl0-multiplication}
\begin{aligned}
&
(\wedge^2\cE \otimes \cL) \otimes (\wedge^2\cE \otimes \cL) \hookrightarrow
\cE \otimes \cE \otimes \cL \otimes \cE \otimes \cE \otimes \cL
\xrightarrow{\ \id_\cE \otimes q \otimes \id_\cE \otimes \id_\cL\ } 
\cE \otimes \cE \otimes \cL \xrightarrow{\ \hphantom{q}\ } 
\wedge^2\cE \otimes \cL,
\quad\text{and}\\
&
(\wedge^2\cE \otimes \cL) \otimes (\wedge^2\cE \otimes \cL) \hookrightarrow
\cE \otimes \cE \otimes \cL \otimes \cE \otimes \cE \otimes \cL
\xrightarrow{\ \id_\cE \otimes q \otimes \id_\cE \otimes \id_\cL\ } 
\cE \otimes \cE \otimes \cL \xrightarrow{\ q\ } 
\cO_S,
\end{aligned}
\end{equation} 
where the first map is the natural embedding,
the second map is induced by the composition
\begin{equation*}
\cE \otimes \cL \otimes \cE 
\xrightarrow{\ \id_\cE \otimes q \otimes \id_\cE } 
\cE \otimes \Sym^2\cE^\vee \otimes \cE \xrightarrow{\qquad} 
\cO,
\end{equation*}
where the second arrow is the natural pairing,
the last map in the second row is analogous,
and the last map in the first row is wedge product.

\begin{example}
\label{ex:clifford}
It is easy to compute the Clifford multiplication~$\Cliff_0(q) \otimes \Cliff_0(q) \to \Cliff_0(q)$ explicitly.
Assume~$S = \Spec(\kk)$ and~$q$ is a diagonal quadratic form in a basis~$e_1,e_2,e_3$ of a vector space~$\cE$, so that
\begin{equation*}
q(x_1e_1 + x_2e_2 + x_3e_3) = a_1x_1^2 + a_2x_2^2 + a_3x_3^2,
\qquad
a_1,a_2,a_3 \in \kk.
\end{equation*}
Then in the basis~$e_{12} = e_1 \wedge e_2$, $e_{13} = e_1 \wedge e_3$, $e_{23} = e_2 \wedge e_3$ of~$\wedge^2\cE$ 
the multiplication takes the form
\begin{equation*}
e_{ij}^2 = -2a_ia_j \cdot \mathbf{1}
\qquad\text{and}\qquad  
e_{ij} \cdot e_{jk} = - e_{jk} \cdot e_{ij} = a_j e_{ik}
\qquad\text{for~$i \ne j \ne k \ne i$}.
\end{equation*}
Thus, the restriction of the commutator 
to the second summand~$\Cliff_0^0(q) \coloneqq \wedge^2\cE \otimes \cL$ of~\eqref{eq:cl0} is the map
\begin{equation}
\label{eq:clifford-commutator}
\wedge^2(\Cliff_0^0(q)) \xrightarrow{\ [-,-]\ } \Cliff_0^0(q),
\qquad 
e_{ij} \wedge e_{jk} \mapsto 2a_je_{ik}.
\end{equation}
Note that the determinant of this map is equal to the determinant of~$q$ up to invertible constant
and in the case where~$q$ is nondegenerate this map is surjective 
and~$\Cliff_0^0(q) = [\Cliff_0(q), \Cliff_0(q)]$. 
\end{example}

\begin{remark}
Assume~$\kk$ is algebraically closed.
The computation of Example~\ref{ex:clifford} shows that
\begin{itemize}
\item 
if~$\rank(q) = 3$ then~$\Cliff_0(q)$ is the algebra of 2-by-2 matrices,
\item 
if~$\rank(q) = 2$ then~$\Cliff_0(q)$ is the path algebra 
of the quiver~$\xymatrix@1{\bullet \ar@/^.5ex/[r]^\alpha & \bullet \ar@/^.5ex/[l]^\beta}$
with relations~$\alpha \cdot \beta = \beta \cdot \alpha = 0$,
\item 
if~$\rank(q) = 1$ then~$\Cliff_0(q)$ is the exterior algebra with~2 generators.
\end{itemize}
It also follows that one can reconstruct the quadratic form~$q$ from the even Clifford algebra.
\end{remark}

To make use of this observation, we introduce the following

\begin{definition}
\label{def:pca}
A locally free~$\cO_S$-algebra~$\cR$ endowed with a direct sum decomposition
\begin{equation}
\label{eq:cr0}
\cR = \cO_S \oplus \cR^0,
\end{equation}
where the first component is generated by the unit of~$\cR$ 
and the second component contains the commutator subsheaf of~$\cR$ (i.e.,~$[\cR,\cR] \subset \cR^0$)
is called {\sf a pointwise Clifford algebra}, 
if for all geometric points~$s \in S$ there is an isomorphism of the algebra~$\cR_s$ 
with the even Clifford algebra of a non-zero quadratic form
such that the restriction~$\cR_s = \kk(s) \oplus \cR_s^0$ of decomposition~\eqref{eq:cr0} 
coincides with decomposition~\eqref{eq:cl0} of the even Clifford algebra.
\end{definition}

\begin{remark}
Specifying a decomposition~\eqref{eq:cr0} is equivalent to choosing a trace map~$\tr \colon \cR \to \cO_S$ 
(i.e., a map vanishing on the commutator subsheaf~$[\cR,\cR] \subset \cR$)
such that~$\tr(\mathbf{1})$ is an invertible section of~$\cO_S$, up to rescaling.
In particular, if~$\cR$ is an Azumaya algebra, then the decomposition~\eqref{eq:cr0} is unique.
\end{remark}

Now we can prove the following generalization of~\cite[Theorem~6.12]{CI}:

\begin{proposition}
\label{prop:pca-cliff}
If~$\cR$ is a pointwise Clifford algebra on a reduced scheme~$S$ then~$\cR \cong \Cliff_0(q_\cR)$ 
for a nowhere vanishing quadratic form~$q_\cR \colon \det(\cR^0) \to \Sym^2(\cR^0)$.

In particular, the operations~$q \mapsto \Cliff_0(q)$ and~$\cR \mapsto q_\cR$
define a bijection between the sets of all nowhere vanishing quadratic forms up to twist
and pointwise Clifford algebras on~$S$.
\end{proposition}

\begin{proof}
Let~$\cR = \cO_S \oplus \cR^0$ be a pointwise Clifford algebra.
Consider the commutator map 
\begin{equation*}
\det(\cR^0) \otimes ({\cR^0})^\vee \cong \wedge^2\cR^0 \xrightarrow{\ [-,-]\ } \cR^0.
\end{equation*}
Since for each geometric point~$s \in S$ the algebra~$\cR_s$ is isomorphic to an even Clifford algebra 
and~$\cR_s^0$ corresponds to the second summand in~\eqref{eq:cl0},
the computation of Example~\ref{ex:clifford} shows that the induced map 
\begin{equation*}
q_\cR \colon \det(\cR^0) \to \cR^0 \otimes \cR^0
\end{equation*}
is symmetric at every geometric point, i.e., 
the composition of~$q_\cR$ with the projection~$\cR^0 \otimes \cR^0 \to \wedge^2(\cR^0)$
vanishes at every geometric point of~$S$.
Since~$S$ is reduced and~$\cR^0$ is locally free, 
it follows that~$q_\cR$ factors as~$\det(\cR^0) \to \Sym^2(\cR^0)$,
hence defines a quadratic form on the vector bundle~$\cE \coloneqq ({\cR^0})^\vee$ with~\mbox{$\cL \coloneqq \det(\cR^0)$}.
It remains to note that the isomorphism of sheaves
\begin{equation*}
\Cliff_0(({\cR^0})^\vee, \det(\cR^0), q_\cR) = 
\cO_S \oplus (\wedge^2({\cR^0})^\vee \otimes \det(\cR^0)) \cong 
\cO_S \oplus \cR^0 = 
\cR
\end{equation*}
is compatible with the Clifford multiplication of the left side and the multiplication of the right side 
at every geometric point~$s \in S$,
hence it is an isomorphism of $\cO_S$-algebras (again, because~$S$ is reduced).

It remains to note that if~$q$ is a nowhere vanishing quadratic form 
and~$\cR \coloneqq \Cliff_0(q)$ is its even Clifford algebra 
then~$\cR^0 = \wedge^2\cE \otimes \cL \cong \cE^\vee \otimes \det(\cE) \otimes \cL$   
and the quadratic form
\begin{equation*}
q_\cR \colon 
\det(\cE)^{\otimes 2} \otimes \cL^{\otimes 3} \cong 
\det(\cR_0) \to 
\Sym^2(\cR_0) \cong
\Sym^2(\cE^\vee \otimes \det(\cE) \otimes \cL)
\end{equation*}
is isomorphic to the normalized twist of~$q$ (see Remark~\ref{rem:normalized-twist}).
Thus, the maps~$q \mapsto \Cliff_0(q)$ and~$\cR \mapsto q_\cR$ give the required bijections.
\end{proof}

\subsection{The kernel category}

Let~$f \colon X \to S$ be a conic bundle.
The kernel category was defined in~\eqref{eq:def-kernel};
it is an admissible triangulated subcategory in~$\Db(X)$.
Similarly, 
\begin{equation*}
\Ker^\perf(f_*) = \Ker^\perf(X/S) \coloneqq \{ \cG \in \Dp(X) \mid f_*(\cG) = 0 \} = \Ker(f_*) \cap \Dp(X)
\end{equation*}
denotes the perfect part of the kernel category.

We briefly review some properties of~$\Ker(f_*)$.

\begin{lemma}
\label{lem:ker-base-change}
The kernel category~$\Ker(f_*) \subset \Db(X)$ of a conic bundle~$f \colon X \to S$ is $S$-linear, i.e.,
\begin{equation*}
\Ker(f_*) \otimes f^*\cH \subset \Ker(f_*)
\qquad\text{and}\qquad
\Ker^\perf(f_*) \otimes f^*\cH \subset \Ker^\perf(f_*)
\qquad\text{for all~$\cH \in \Dp(S)$.}
\end{equation*}
Moreover, if~$\phi \colon S' \to S$ is any morphism 
and~$f' \colon X' \coloneqq X \times_S S' \to S'$ is the base change of~$X \to S$ then
\begin{equation}
\label{eq:ker-bc}
\phi_{X*}(\Ker(f'_*)) \subset \Ker(f_*)
\qquad\text{and}\qquad
\phi_X^*(\Ker^\perf(f_*)) \subset \Ker^\perf(f'_*),
\end{equation}
where~$\phi_X \colon X' \to X$ is the morphism induced by~$\phi$.
In particular, if~$\cG \in \Ker^\perf(f_*)$ then 
\begin{equation}
\label{eq:hb-ker}
\rH^\bullet(X_s, \cG\vert_{X_s}) = 0
\end{equation}
for each geometric point~$s \in S$.
\end{lemma}

\begin{proof}
The $S$-linear property follows immediately 
from the projection formula~$f_*(\cG \otimes f^*\cH) \cong f_*(\cG) \otimes \cH$.
Furthermore, $f_* \circ \phi_{X*} \cong \phi_* \circ f'_*$ and since~$f$ is flat, we have
base change isomorphisms
\begin{equation*}
f'_* \circ \phi_X^* \cong \phi^* \circ f_* \colon \Dp(X) \to \Dp(S'),
\end{equation*}
and~\eqref{eq:ker-bc} follows.
Applying the second inclusion for~$S' = \{s\}$, we obtain~\eqref{eq:hb-ker}.
\end{proof}

\begin{lemma}
\label{lem:ker-t-structure}
The subcategory~$\Ker(f_*) \subset \Db(X)$ is closed under the truncation functors of the standard t-structure on~$\Db(X)$.
In other words, $f_*(\cG) = 0$ if and only if~$f_*(\cH^i(\cG)) = 0$ for all~$i \in \ZZ$,
where~$\cH^i(\cG)$ is the $i$-th cohomology sheaf of~$\cG$.
Moreover, the subcategories
\begin{equation*}
\Ker(f_*)^{\le 0} \coloneqq \Ker(f_*) \cap \bD(X)^{\le 0}
\qquad\text{and}\qquad 
\Ker(f_*)^{\ge 0} \coloneqq \Ker(f_*) \cap \bD(X)^{\ge 0}
\end{equation*}
define a t-structure on~$\Ker(f_*)$ such that the embedding functor~$\Ker(f_*) \hookrightarrow \Db(X)$ is t-exact.
\end{lemma}

\begin{proof}
To prove that~$\Ker(f_*)$ is closed under truncations, 
note that the dimension of the fibers of~$f$ equals~$1$, hence the hypercohomology spectral sequence
\begin{equation*}
\bR^jf_*(\cH^i(\cG)) \Rightarrow \cH^{i+j}(f_*\cG)
\end{equation*}
degenerates in the second term and the vanishing of~$f_*\cG$ is equivalent to the vanishing of~$f_*(\cH^i(\cG))$ for all~$i$.
This in turn implies that the truncation functors of the standard t-structure preserve~$\Ker(f_*)$
and that the standard t-structure induces a t-structure on~$\Ker(f_*)$.
\end{proof}

By~\cite{K08} the kernel category is a component of the $S$-linear semiorthogonal decomposition
\begin{equation}
\label{eq:dbx}
\Db(X) = \langle \Ker(f_*), f^*(\Db(S)) \rangle
\end{equation}
and is equivalent to the bounded derived category~$\Db(S, \Cliff_0(q))$ 
of coherent right modules over~$\Cliff_0(q)$.
To explain the construction of this equivalence, recall from~\cite[\S3.3]{K08} that
the even Clifford algebra~$\Cliff_0(q)$ comes with a sequence of natural $\Cliff_0(q)$-bimodules:
\begin{equation}
\label{eq:clifford-all}
\Cliff_1(q) \coloneqq \cE \oplus (\wedge^3\cE \otimes \cL),
\qquad
\Cliff_{2i}(q) \coloneqq \Cliff_0(q) \otimes \cL^{-i},
\qquad 
\Cliff_{2i+1}(q) \coloneqq \Cliff_1(q) \otimes \cL^{-i},
\end{equation}
where the bimodule structure 
is defined by the Clifford multiplication~$\Cliff_i(q) \otimes \Cliff_j(q) \to \Cliff_{i+j}(q)$
defined analogously to~\eqref{eq:cl0-multiplication}, see~\cite[\S3]{K14}.
Moreover, by~\cite[Lemma~3.8 and Corollary~3.9]{K08} we have
\begin{equation}
\label{eq:cli-clj}
\Cliff_i(q) \otimes_{\Cliff_0(q)} \Cliff_j(q) \cong \Cliff_{i+j}(q)
\qquad\text{and}\qquad 
\cRHom_{\Cliff_0(q)}(\Cliff_i(q), \Cliff_j(q)) \cong \Cliff_{j-i}(q).
\end{equation} 
In particular, all~$\Cliff_i(q)$ are locally projective left or right modules over~$\Cliff_0(q)$.

We denote by~$\io \colon X \hookrightarrow \P_S(\cE)$ the embedding
and by~$p \colon \P_S(\cE) \to S$ the projection, so that~$f = p \circ \io$.

\begin{theorem}[{\cite[Lemmas~4.5 and~4.7, Proposition~4.9, and Theorem~4.2]{K08}}]
\label{thm:k08}
For each~$i \in \ZZ$ there is a left $f^*\Cliff_0(q)$-module~$\cF^i_{X/S}$ on~$X$,
which is locally free of rank~$2$ over~$\cO_X$,
and an exact sequence 
\begin{equation}
\label{eq:cf-io}
0 \to 
p^*\Cliff_{i-1}(q) \otimes \cO_{\P_S(\cE)/S}(-2) \xrightarrow{\ \ \varphi_i\ \ } 
p^*\Cliff_i(q) \otimes \cO_{\P_S(\cE)/S}(-1) \xrightarrow{\qquad} 
\io_*\cF^i_{X/S} \to 0
\end{equation}
of left $p^*\Cliff_0(q)$-modules,
where the morphism~$\varphi_i$ 
is induced by the embedding~\mbox{$\cO_{X/S}(-1) \hookrightarrow p^*\cE \hookrightarrow p^*\Cliff_1(q)$}
and the Clifford multiplication~$\Cliff_{i-1}(q) \otimes \Cliff_1(q) \to \Cliff_i(q)$.
Moreover, the functor
\begin{equation}
\label{eq:phi-i}
\Phi_{\cF^i_{X/S}} \colon \Db(S, \Cliff_0(q)) \to \Db(X),
\qquad 
\cM \mapsto f^*\cM \otimes_{f^*\Cliff_0(q)} \cF^i_{X/S},
\end{equation}
is $S$-linear, t-exact, fully faithful and defines an equivalence~$\Db(S, \Cliff_0(q)) \simeq \Ker(f_*) \subset \Db(X)$,
while
\begin{equation}
\label{eq:phi-i-shriek}
\Phi^!_{\cF^i_{X/S}} \colon \Db(X) \to \Db(S, \Cliff_0(q)),
\qquad 
\cG \mapsto f_*\cRHom(\cF^i_{X/S}, \cG)
\end{equation}
is its right adjoint.
In particular, the restriction~$\Phi^!_{\cF^i_{X/S}}\vert_{\Ker(f_*)}$ is the inverse of~$\Phi_{\cF^i_{X/S}}$.
\end{theorem}

We will say that a bundle~$\cF \in \Ker(f_*)$ {\sf compactly generates~$\Ker(f_*)$ over~$S$} if
\begin{equation*}
(\cF \otimes f^*\Dp(S))^\perp \cap \Ker(f_*) = 0.
\end{equation*}
When~$S = \Spec(\kk)$, this notion agrees with the usual notion of compact generation.

\begin{corollary}
\label{cor:cfi-fs}
For each~$i \in \ZZ$ we have~$\cF^i_{X/S} \cong \Phi_{\cF^0_{X/S}}(\Cliff_i(q))$ and
\begin{equation}
\label{eq:fs-end-cfi}
f_*\cRHom(\cF^i_{X/S}, \cF^j_{X/S}) \cong \Cliff_{j-i}(q).
\end{equation} 
Moreover, $\cF^i_{X/S} \in \Ker(f_*)$ and it compactly generates~$\Ker(f_*)$ over~$S$.
\end{corollary}

\begin{proof}
First, using the projection formula and the definition of~$\Phi_{\cF^0_{X/S}}$ we obtain
\begin{equation*}
\io_*\Phi_{\cF^0_{X/S}}(\Cliff_i(q)) \cong
\io_*(f^*\Cliff_i(q) \otimes_{f^*\Cliff_0(q)} \cF^0_{X/S}) \cong
p^*\Cliff_i(q) \otimes_{p^*\Cliff_0(q)} \io_*\cF^0_{X/S}.
\end{equation*}
Next, tensoring the resolution~\eqref{eq:cf-io} of~$\io_*\cF^0_{X/S}$ by~$p^*\Cliff_i(q)$,
using~\eqref{eq:cli-clj}, and comparing the result with~\eqref{eq:cf-io}
we see that~$\io_*\Phi_{\cF^0_{X/S}}(\Cliff_i(q)) \cong \io_*\cF^i_{X/S}$,
and the first statement follows.
In particular, $\cF^i_{X/S} \in \Ker(f_*)$.
Since~$\Phi_{\cF^0_{X/S}}$ is an $S$-linear equivalence~$\Db(S, \Cliff_0(q)) \to \Ker(f_*)$, it also follows that
\begin{equation*}
f_*\cRHom(\cF^i_{X/S}, \cF^j_{X/S}) \cong 
\cRHom_{\Cliff_0(q)}(\Cliff_i(q), \Cliff_j(q)) \cong
\Cliff_{j-i}(q),
\end{equation*}
where the second isomorphism is~\eqref{eq:cli-clj}.
Compact generation follows immediately from Theorem~\ref{thm:k08}
because~$\Cliff_0(q)$ is a compact generator for~$\Db(S, \Cliff_0(q))$ over~$S$.
\end{proof}

We also note that~\eqref{eq:dbx} implies a similar semiorthogonal decomposition for the perfect categories.

\begin{corollary}
\label{cor:ker-perf}
If~$f \colon X \to S$ is a conic bundle, there is a semiorthogonal decomposition
\begin{equation}
\label{eq:ker-perf}
\Dp(X) = \langle \Ker^\perf(f_*), f^*(\Dp(S)) \rangle.
\end{equation} 
\end{corollary}

\begin{proof}
The claim is local over~$S$, so we may assume~$S$ quasiprojective.
Moreover, since~$f$ is proper and Gorenstein, the functor~$f^*$ has both adjoints on the bounded category,
hence~\eqref{eq:dbx} is a strong semiorthogonal decomposition in the sense of~\cite[Definition~2.6]{K11}.
Thus, we can apply~\cite[Proposition~4.1]{K11}.
\end{proof}

\subsection{Abstract spinor bundles}

Recall Definition~\ref{def:abs} from the introduction.
The following proposition shows what the restrictions of an abstract spinor bundle 
to the geometric fibers of a conic bundle look like.
In the statement we relax the conditions defining an abstract spinor bundle.

\begin{proposition}
\label{prop:asb-fibers}
Let~$X$ be a conic over an algebraically closed field~$\kk$ 
and let~$\cF$ be an acyclic vector bundle of rank~$2$ on~$X$ such that either
\begin{enumerate}[label={\textup{(\roman*)}}]
\item 
\label{it:asb-c1}
$-\rc_1(\cF)$ is ample, or
\item
\label{it:asb-gen}
$\cF$ compactly generates~$\Ker(X/\kk)$.
\end{enumerate}
Then~$\det(\cF) \cong \omega_X$ and the isomorphism class of~$\cF$ is uniquely determined.
Explicitly,
\begin{alenumerate}
\item 
\label{it:r3}
if~$X$ is smooth, so that~$X \cong \P^1$, then~$\cF \cong \cO_X(-1)^{\oplus 2}$,
\item 
\label{it:r2}
if~$X$ is reducible, so that~$X = X' \cup_{x_0} X'' \cong \P^1 \cup_{x_0} \P^1$, then there is an exact sequence
\begin{equation}
\label{eq:asb-reducible}
0 \to \cF \to (\cO_{X'} \oplus \cO_{X'}(-1)) \oplus (\cO_{X''} \oplus \cO_{X''}(-1)) \to \cO_{x_0}^{\oplus 2} \to 0,
\end{equation}
where the second arrow is given by the matrix~$\left(\begin{smallmatrix} 1 & 0 & 0 & 1 \\ 0 & 1 & 1 & 0 \end{smallmatrix}\right)$,
\item 
\label{it:r1}
if~$X$ is non-reduced, so that~$X_{\mathrm{red}} \cong \P^1$, then there is an exact sequence
\begin{equation}
\label{eq:asb-nonreduced}
0 \to 
\upxi_*(\cO_{X_{\mathrm{red}}}(-1) \oplus \cO_{X_{\mathrm{red}}}(-2)) \to 
\cF \to
\upxi_*(\cO_{X_{\mathrm{red}}} \oplus \cO_{X_{\mathrm{red}}}(-1)) \to 
0,
\end{equation}
where~$\upxi \colon X_{\mathrm{red}} \hookrightarrow X$ is the natural embedding,
and the connecting morphisms
\begin{align}
\label{eq:connecting-io}
\bL_1\upxi^*\upxi_*(\cO_{X_{\mathrm{red}}} \oplus \cO_{X_{\mathrm{red}}}(-1)) &\to 
\bL_0\upxi^*\upxi_*(\cO_{X_{\mathrm{red}}}(-1) \oplus \cO_{X_{\mathrm{red}}}(-2)),
\\
\label{eq:connecting-h}
\rH^0(X_{\mathrm{red}}, \cO_{X_{\mathrm{red}}} \oplus \cO_{X_{\mathrm{red}}}(-1)) &\to 
\rH^1(X_{\mathrm{red}}, \cO_{X_{\mathrm{red}}}(-1) \oplus \cO_{X_{\mathrm{red}}}(-2))
\end{align} 
are isomorphisms.
\end{alenumerate}
\end{proposition}

\begin{proof}
\ref{it:r3}
If~$X$ is smooth, the isomorphism~$\cF \cong \cO_X(-1)^{\oplus 2}$ follows easily from acyclicity of~$\cF$.

\ref{it:r2}
Assume~$X$ is reducible.
Then we have an exact sequence~$0 \to \cO_{X'}(-1) \to \cO_X \to \cO_{X''} \to 0$
and a similar sequence~$0 \to \cO_{X''}(-1) \to \cO_X \to \cO_{X'} \to 0$.
Tensoring by~$\cF$ we obtain exact sequences
\begin{equation*}
0 \to \cF\vert_{X'}(-1) \to \cF \to \cF\vert_{X''} \to 0
\qquad\text{and}\qquad 
0 \to \cF\vert_{X''}(-1) \to \cF \to \cF\vert_{X'} \to 0.
\end{equation*}
As~$\cF$ is acyclic and~$\dim(X) = 1$, the cohomology exact sequences imply
\begin{align*}
\rH^0(X',\cF\vert_{X'}(-1)) =
\rH^0(X'',\cF\vert_{X''}(-1)) =
\rH^1(X',\cF\vert_{X'}) =
\rH^1(X'',\cF\vert_{X''}) = 
0,
\\
\rH^0(X'',\cF\vert_{X''}) \cong 
\rH^1(X',\cF\vert_{X'}(-1)),
\qquad\ \,\,\,
\rH^0(X',\cF\vert_{X'}) \cong
\rH^1(X'',\cF\vert_{X''}(-1)).
\end{align*}
On the other hand, since~$\cF$ is a vector bundle of rank~2 and~$X' \cong X'' \cong \P^1$, we can write
\begin{equation*}
\cF\vert_{X'} \cong \cO_{X'}(-a') \oplus \cO_{X'}(-b')
\qquad\text{and}\qquad 
\cF\vert_{X''} \cong \cO_{X''}(-a'') \oplus \cO_{X''}(-b''),
\end{equation*}
and the cohomology equalities imply that~$0 \le a',b',a'',b'' \le 1$ and~$a' + b' + a'' + b'' = 2$.
We prove that the cases where~$a' = b' = 0$ or~$a'' = b'' = 0$ are not possible.
Indeed, if~\ref{it:asb-c1} holds, i.e., if~$-\rc_1(\cF)$ is ample, 
then~$a' + b', a'' + b'' > 0$, as required.
Similarly, if~\ref{it:asb-gen} holds and~$a' = b' = 0$ or~$a'' = b'' = 0$ 
then~$\cF$ is left orthogonal to~$\cO_{X'}(-1)$ or~$\cO_{X''}(-1)$, respectively. 
Thus, we must have~$\{a',b'\} = \{a'',b''\} = \{0,1\}$.

Now consider the exact sequence~$0 \to \cO_X \to \cO_{X'} \oplus \cO_{X''} \to \cO_{x_0} \to 0$.
Tensoring it by~$\cF$ and using the above computation of~$\cF\vert_{X'}$ and~$\cF\vert_{X''}$, 
we obtain~\eqref{eq:asb-reducible}.
Now the acyclicity of~$\cF$ implies the surjectivity 
of the morphism~\mbox{$\cO_{X'} \oplus \cO_{X''} \to \cO_{x_0}^{\oplus 2}$}, 
hence there is a basis in the target such that this morphism 
is given by the matrix~$\left(\begin{smallmatrix} 1 & 0 \\ 0 & 1 \end{smallmatrix}\right)$.
Finally, using local freeness of~$\cF$ and  acting by appropriate automorphisms of~$\cF\vert_{X'}$ and~$\cF\vert_{X''}$,
it is easy to reduce the second arrow of~\eqref{eq:asb-reducible} to the required form.

The uniqueness of~$\cF$ follows from~\eqref{eq:asb-reducible}.

\ref{it:r1}
Assume~$X$ is non-reduced.
Tensoring the exact sequence~$0 \to \upxi_*\cO_{X_{\mathrm{red}}}(-1) \to \cO_X \to \upxi_*\cO_{X_{\mathrm{red}}} \to 0$
by~$\cF$ we obtain an exact sequence
\begin{equation*}
0 \to \upxi_*(\cF\vert_{X_{\mathrm{red}}}(-1)) \to \cF \to \upxi_*(\cF\vert_{X_{\mathrm{red}}}) \to 0.
\end{equation*}
As in~\ref{it:r2}, we can write~$\cF\vert_{X_{\mathrm{red}}} \cong \cO_{X_{\mathrm{red}}}(-a) \oplus _{X_{\mathrm{red}}}(-b)$,
and arguing analogously, we conclude that~$a = 0$ and~$b = 1$ and obtain~\eqref{eq:asb-nonreduced}.

Applying the functor~$\bL_\bullet\upxi^*$ to~\eqref{eq:asb-nonreduced}, we obtain a long exact sequence
\begin{equation*}
\dots \to
\bL_1\upxi^*(\cF) \to 
\bL_1\upxi^*\upxi_*(\cF\vert_{X_{\mathrm{red}}}) \to 
\bL_0\upxi^*\upxi_*(\cF\vert_{X_{\mathrm{red}}}(-1)) \to 
\bL_0\upxi^*(\cF) \to 
\bL_0\upxi^*\upxi_*(\cF\vert_{X_{\mathrm{red}}}) \to 
0.
\end{equation*}
The last arrow here is an isomorphism, and the term~$\bL_1\upxi^*(\cF)$ is zero (because~$\cF$ is locally free),
hence the connecting morphism~\eqref{eq:connecting-io} is an isomorphism.
Similarly, applying the functor~$\rH^\bullet(-)$ to~\eqref{eq:asb-nonreduced} and using the acyclicity of~$\cF$, 
we conclude that~\eqref{eq:connecting-h} is an isomorphism as well.

To finish the proof it remains to show that an extension~\eqref{eq:asb-nonreduced} 
for which the connecting morphisms~\eqref{eq:connecting-io} and~\eqref{eq:connecting-h} are isomorphisms
is unique up to isomorphism.
For this note that there is an exact sequence
\begin{equation*}
0 \to 
\Ext^1(\cF\vert_{X_{\mathrm{red}}}, \cF\vert_{X_{\mathrm{red}}}(-1)) \to
\Ext^1(\upxi_*(\cF\vert_{X_{\mathrm{red}}}), \upxi_*(\cF\vert_{X_{\mathrm{red}}}(-1))) \to
\Hom(\cF\vert_{X_{\mathrm{red}}}(-1), \cF\vert_{X_{\mathrm{red}}}(-1)) \to
0.
\end{equation*}
The second arrow here takes an extension to the connecting morphism~\eqref{eq:connecting-io},
and the first arrow is the action of~$\upxi_*$.
Note also that~$\Ext^1(\cF\vert_{X_{\mathrm{red}}}, \cF\vert_{X_{\mathrm{red}}}(-1)) =
\Ext^1(\cO \oplus \cO(-1), \cO(-1) \oplus \cO(-2)) = \kk$
and the composition
\begin{equation*}
\Ext^1(\cF\vert_{X_{\mathrm{red}}}, \cF\vert_{X_{\mathrm{red}}}(-1)) \to
\Ext^1(\upxi_*(\cF\vert_{X_{\mathrm{red}}}), \upxi_*(\cF\vert_{X_{\mathrm{red}}}(-1))) \to
\Hom(\rH^0(\cF\vert_{X_{\mathrm{red}}}), \rH^1(\cF\vert_{X_{\mathrm{red}}}(-1))) 
\end{equation*}
(where the second arrow takes an extension to the connecting morphism~\eqref{eq:connecting-h}) is an isomorphism.
This means that the morphism
\begin{equation*}
\Ext^1(\upxi_*(\cF\vert_{X_{\mathrm{red}}}), \upxi_*(\cF\vert_{X_{\mathrm{red}}}(-1))) \to
\Hom(\cF\vert_{X_{\mathrm{red}}}(-1), \cF\vert_{X_{\mathrm{red}}}(-1)) \oplus
\Hom(\rH^0(\cF\vert_{X_{\mathrm{red}}}), \rH^1(\cF\vert_{X_{\mathrm{red}}}(-1))) 
\end{equation*}
that takes an extension to the connecting morphisms~\eqref{eq:connecting-io} and~\eqref{eq:connecting-h}) is an isomorphism.
It remains to note that the automorphisms of the first and last terms of~\eqref{eq:asb-nonreduced}
act transitively on pairs~$(\phi_1,\phi_2)$ of isomorphisms in the target of the above map,
hence the isomorphism class of~$\cF$ is unique.

The isomorphism~$\det(\cF) \cong \omega_X$ in cases~\ref{it:r3}, \ref{it:r2}, and~\ref{it:r1} follows easily.
\end{proof}

The following corollary shows that the second assumption in Definition~\ref{def:abs} may be relaxed.

\begin{corollary}
\label{cor:gen-asb}
If~$\cF$ is a vector bundle of rank~$2$ on a conic bundle~$X/S$ such that~$f_*(\cF) = 0$ and
\begin{enumerate}[label={\textup{(\roman*)}}]
\item 
\label{def2:asb-c1}
$-\rc_1(\cF)$ is $f$-ample, or
\item
\label{def2:asb-gen}
$\cF$ compactly generates~$\Ker(f_*)$ over~$S$,
\end{enumerate}
then~$\rc_1(\cF) = K_{X/S}$ in~$\Pic(X/S)$.
In particular, $\cF$ is an abstract spinor bundle.
\end{corollary}

\begin{proof}
Indeed, if~$\cF$ satisfies~\ref{def2:asb-c1} or~\ref{def2:asb-gen}
then the restriction of~$\cF$ to any geometric fiber of~$X/S$ satisfies one of the assumptions of Proposition~\ref{prop:asb-fibers},
hence~$\det(\cF)$ is isomorphic to~$\omega_{X/S}$ on each geometric fiber, 
hence~$\rc_1(\cF) = K_{X/S}$ in~$\Pic(X/S)$,
and therefore~$\cF$ is an abstract spinor bundle.
\end{proof}

\subsection{Canonical spinor bundles}

The following is an immediate consequence of Theorem~\ref{thm:k08}, Corollaries~\ref{cor:cfi-fs}, \ref{cor:gen-asb}, 
the isomorphisms~\eqref{eq:clifford-all}, and the definitions.

\begin{lemma}
\label{lem:csb-asb}
The sheaves~$\cF^i_{X/S}$ on~$X$ defined by~\eqref{eq:cf-io} are abstract spinor bundles.
Moreover,
\begin{equation}
\label{eq:cfi-all}
\cF^{i-2}_{X/S} \cong \cF^i_{X/S} \otimes f^*\cL
\end{equation} 
and the restrictions of~$\cF^i_{X/S}$ to geometric fibers of~$X/S$ 
are the sheaves described in Proposition~\textup{\ref{prop:asb-fibers}}.
\end{lemma}

The bundles~$\cF^i_{X/S}$ are called the {\sf canonical spinor bundles} on~$X/S$.
The following lemma gives an explicit description for~$\cF^0_{X/S}$ and~$\cF^{-1}_{X/S}$.
By~\eqref{eq:cfi-all} all other~$\cF^i_{X/S}$ can be obtained from these two by twists. 
In the computation we use the following exact sequence
\begin{equation}
\label{eq:koszul}
0 \to \cO_{X/S} \to 
f^*\cE \otimes \cO_{X/S}(1) \to 
f^*(\wedge^2\cE) \otimes \cO_{X/S}(2) \to 
f^*(\wedge^3\cE) \otimes \cO_{X/S}(3) \to 
0,
\end{equation}
obtained by restricting the Koszul complex from~$\P_S(\cE)$ to~$X$.

\begin{lemma}
\label{lem:cf0-cf1}
There are unique abstract spinor bundles~$\cF^0$ and~$\cF^{-1}$ on~$X$ fitting into exact sequences
\begin{align}
\label{eq:csb}
0 \to \omega_{X/S} \to \cF^0 \to \cO_X \to 0,
\\
\label{eq:csb-1}
0 \to \cF^{-1} \to f^*\cE^\vee \to \cO_{X/S}(1) \to 0.
\end{align} 
Moreover, $\cF^i_{X/S} \cong \cF^i \otimes f^*(\det(\cE) \otimes \cL)$ for~$i \in \{-1,0\}$.
\end{lemma}

In particular, if the quadratic form~$q$ is normalized (see Remark~\ref{rem:normalized-twist}), 
we have~$\cF^i_{X/S} \cong \cF^i$ for~$i \in \{-1,0\}$.
The extension~$\cF^0$ was also considered in~\cite[Definition~6.2]{CI}.

\begin{proof}
Since~$f_*\omega_{X/S} \cong \cO_S[-1]$, there is a unique extension~\eqref{eq:csb} such that~$f_*(\cF^0) = 0$
(cf.~\cite[Lemma~6.3]{CI}).
Further, \eqref{eq:csb} implies~$\rc_1(\cF^0) = K_{X/S}$ in~$\Pic(X)$, hence~$\cF^0$ is an abstract spinor bundle.

Similarly, since~$f_*\cO_{X/S}(1) \cong \cE^\vee$, there is a unique exact sequence~\eqref{eq:csb-1} such that~$f_*(\cF^{-1}) = 0$.
Moreover, \eqref{eq:csb-1} implies that~$\rc_1(\cF^{-1}) = K_{X/S}$ in~$\Pic(X/S)$, hence~$\cF^{-1}$ is an abstract spinor bundle.

To relate~$\cF^0_{X/S}$ to~$\cF^0$, consider the composition
\begin{equation*}
f^*\Cliff_0(q) \otimes \cO_{X/S}(-1) \hookrightarrow
f^*\Cliff_0(q) \otimes f^*\cE \twoheadrightarrow
f^*((\wedge^2\cE \otimes \cL) \otimes \cE) \twoheadrightarrow
f^*(\wedge^3\cE \otimes \cL),
\end{equation*}
where the first arrow is the tautological embedding,
the second is the projection of~$\Cliff_0(q)$ onto the second component of~\eqref{eq:cl0},
and the third is wedge product.
Its restriction to the summand~$f^*(\wedge^2\cE \otimes \cL) \otimes \cO_{X/S}(-1)$ of~$f^*\Cliff_0(q) \otimes \cO_{X/S}(-1)$ 
coincides with the last morphism in the Koszul complex~\eqref{eq:koszul},
which is surjective, hence the above composition is surjective as well.
On the other hand, the composition
\begin{equation}
\label{eq:composition}
\xymatrix@1@C=4em{
f^*\Cliff_{-1}(q) \otimes \cO_{X/S}(-2) \ar[r]^{\ \io^*(\varphi_0)\ } &
f^*\Cliff_0(q) \otimes \cO_{X/S}(-1) \ar@{->>}[r] &
f^*(\wedge^3\cE \otimes \cL)}
\end{equation}
with the first morphism in~\eqref{eq:cf-io} restricted to~$X$ vanishes.
Indeed, on the summand~$\cE \otimes \cL \subset \Cliff_{-1}(q)$ at any geometric point~$x \in X$ 
(where we consider~$x$ as a vector in~$\cE_{f(x)}$ via the embedding~$X \subset \P_S(\cE)$) 
the composition~\eqref{eq:composition} is given by
\begin{equation*}
e \mapsto 
e \wedge x + q(e,x) \cdot \mathbf{1} \mapsto 
e \wedge x \wedge x = 0.
\end{equation*}
Similarly, on the summand~$\wedge^3\cE \otimes \cL^{\otimes 2} \subset \Cliff_{-1}(q)$ at point~$x$, 
the composition~\eqref{eq:composition} is equal to
\begin{align*}
e_1 \wedge e_2 \wedge e_3 &\mapsto 
q(e_1,x)\,e_2 \wedge e_3 - 
q(e_2,x)\,e_1 \wedge e_3 + 
q(e_3,x)\,e_1 \wedge e_2 
\\
& \mapsto
q(e_1,x)\,x \wedge e_2 \wedge e_3 - 
q(e_2,x)\,x \wedge e_1 \wedge e_3 + 
q(e_3,x)\,x \wedge e_1 \wedge e_2
\end{align*}
(where~$(e_1,e_2,e_3)$ is a basis in the fiber~$\cE_{f(x)}$ of~$\cE$),
hence the right side is equal to~$q(x,x)\,e_1 \wedge e_2 \wedge e_3 = 0$.

Thus, \eqref{eq:composition} vanishes.
On the other hand, restricting~\eqref{eq:cf-io} to~$X$ 
we see that~$\cF^0_{X/S}$ is the cokernel of the first arrow~$\io^*(\varphi_0)$ in~\eqref{eq:composition},
hence we obtain an epimorphism~$\cF^0_{X/S} \twoheadrightarrow f^*(\wedge^3\cE \otimes \cL)$.
Furthermore, by Lemma~\ref{lem:csb-asb} and Definition~\ref{def:abs} we have~$\rc_1(\cF^0_{X/S}) = K_{X/S}$ in~$\Pic(X/S)$,
hence the kernel of this epimorphism  can be written as~$\omega_{X/S} \otimes f^*\cL'$ 
for a line bundle~$\cL'$ on~$S$, 
hence we have an exact sequence
\begin{equation*}
0 \to \omega_{X/S} \otimes f^*\cL' \to \cF^0_{X/S} \to f^*(\wedge^3\cE \otimes \cL) \to 0.
\end{equation*}
Pushing it forward to~$S$ and using the projection formula, 
the isomorphism~$f_*\omega_{X/S} \cong \cO_S[-1]$,
and the fact that~$\cF^0_{X/S} \in \Ker(f_*)$ by Corollary~\ref{cor:cfi-fs},
we see that~$\cL' \cong \wedge^3\cE \otimes \cL$,
and conclude that~$\cF^0_{X/S}$ is the twist of the sheaf~$\cF^0$ 
defined in~\eqref{eq:csb} by~$f^*(\wedge^3\cE \otimes \cL)$.

Similarly, to relate~$\cF^{-1}_{X/S}$ to~$\cF^{-1}$, we note that the composition
\begin{equation*}
f^*\cO_S \otimes \cO_{X/S}(-2) \hookrightarrow 
f^*\Cliff_0(q) \otimes \cO_{X/S}(-2) \xrightarrow{\ \io^*(\varphi_1)\ }
f^*\Cliff_1(q) \otimes \cO_{X/S}(-1) 
\end{equation*}
at a geometric point~$x \in X$ is given by~$\mathbf{1} \mapsto x \in \cE_{f(x)} \subset \Cliff_1(q)_{f(x)}$,
hence the morphism~$\io^*(\varphi_1)$ preserves the filtrations 
of~$f^*\Cliff_0(q) \otimes \cO_{X/S}(-2)$ and~$f^*\Cliff_1(q) \otimes \cO_{X/S}(-1)$
induced by the exact sequences
\begin{equation*}
0 \to \cO_S \to \Cliff_0(q) \to \wedge^2\cE \otimes \cL \to 0
\qquad\text{and}\qquad 
0 \to \cE \to \Cliff_1(q) \to \wedge^3\cE \otimes \cL \to 0.
\end{equation*}
Furthermore, a simple computation shows that the maps induced by~$\io^*(\varphi_1)$ on the factors are the maps
\begin{equation*}
\cO_{X/S}(-2) \to f^*\cE \otimes \cO_{X/S}(-1)
\qquad\text{and}\qquad 
f^*(\wedge^2\cE \otimes \cL) \otimes \cO_{X/S}(-2) \to f^*(\wedge^3\cE \otimes \cL) \otimes \cO_{X/S}(-1)
\end{equation*}
obtained by twist from the Koszul complex~\eqref{eq:koszul}.
It follows that the sheaf~$\cF^{1}_{X/S} \cong \Coker(\io^*(\varphi_1))$ (see~\eqref{eq:cf-io})
is isomorphic to the cokernel of the first arrow in the Koszul complex twisted by~$\cO_{X/S}(-2)$,
i.e., to the kernel of the third arrow twisted by~$\cO_{X/S}(-2)$; explicitly
\begin{equation*}
\cF^{1}_{X/S} \cong 
\Ker\Big(f^*(\wedge^2\cE) \to f^*(\wedge^3\cE) \otimes \cO_{X/S}(1)\Big) \cong
\Ker\Big(f^*\cE^\vee \to \cO_{X/S}(1)\Big) \otimes f^*(\wedge^3\cE) \cong
\cF^{-1} \otimes f^*(\wedge^3\cE).
\end{equation*}
Twisting this by~$f^*\cL$ and using~\eqref{eq:cfi-all}, we obtain the last claim.
\end{proof}

\subsection{Proof of Theorem~\ref{thm:asb-cb} and Corollary~\ref{cor:he-sm}}

Recall Definition~\ref{def:pca} of pointwise Clifford algebras.

\begin{proposition}
\label{prop:cf-cliff}
If~$\cF$ is an abstract spinor bundle 
then~$\cF$ is a tilting generator for the category~$\Ker(f_*)$ over~$S$.
Moreover, $\cR_\cF \coloneqq f_*\cEnd(\cF)$
is a pointwise Clifford algebra on~$S$ and the adjoint functors
\begin{equation}
\label{eq:phi-phi}
\begin{aligned}
\Phi_\cF &\colon \Db(S, \cR_\cF) \to \Db(X),
&\qquad 
\cH &\mapsto f^*\cH \otimes_{f^*\cR_\cF} \cF,
\\
\Phi^!_\cF & \colon \Db(X) \to \Db(S, \cR_\cF),
&\qquad 
\cG &\mapsto f_*\cRHom(\cF, \cG)
\end{aligned}
\end{equation} 
define an $S$-linear t-exact equivalence~$\Db(S,\cR_\cF) \simeq \Ker(f_*)$ such that~$\Phi_\cF(\cR_\cF) \cong \cF$.
\end{proposition}

\begin{proof}
Let~$\cF^0_{X/S}$ be the canonical spinor bundle of~$X/S$. 
Note that
\begin{equation}
\label{eq:asb-csb-fibers}
\cF\vert_{X_s} \cong \cF^0_{X/S}\vert_{X_s} 
\end{equation} 
for any geometric point~$s \in S$,
because both~$\cF^0_{X/S}$ and~$\cF$ satisfy the assumptions of Proposition~\ref{prop:asb-fibers}.
We deduce from~\eqref{eq:asb-csb-fibers} that the functor~$\Phi_\cF^!$ agrees pointwise 
with the functor~$\Phi^!_{\cF^0_{X/S}}$ defined in~\eqref{eq:phi-i-shriek}, i.e., 
\begin{equation}
\label{eq:phi-cf-phi-cfi}
\Phi^!_\cF(\cG)_s \cong \Phi^!_{\cF^0_{X/S}}(\cG)_s
\end{equation}
for all~$\cG \in \Ker(f_*)$ and geometric points~$s \in S$.
Indeed, since~$f$ is flat, base change isomorphisms imply
\begin{align*}
\Phi^!_\cF(\cG)_s = 
(f_*\cRHom(\cF, \cG))_s &\cong 
\rH^\bullet(X_s, \cRHom(\cF\vert_{X_s}, \cG\vert_{X_s})) \\ &\cong
\rH^\bullet(X_s, \cRHom(\cF^0_{X/S}\vert_{X_s}, \cG\vert_{X_s})) \cong
(f_*\cRHom(\cF^0_{X/S}, \cG))_s =
\Phi^!_{\cF^0_{X/S}}(\cG)_s,
\end{align*}
as required.

Next, we check that the functor~$\Phi^!_\cF$ is t-exact on~$\Ker(f_*)$.
Indeed, since~$\cF$ is locally free, $\Phi^!_\cF$ is left exact, so it remains to show it is right exact.
If it is not, Lemma~\ref{lem:ker-t-structure} implies that there is a sheaf~$\cG \in \Ker(f_*)$ and a geometric point~$s \in S$
such that~\mbox{$(\Phi^!_\cF(\cG))_s \not\in \bD(S)^{\le 0}$}.
But then~\eqref{eq:phi-cf-phi-cfi} implies that~$\Phi^!_{\cF^0_{X/S}}(\cG)_s \not\in \bD(S)^{\le 0}$
in contradiction to t-exactness of the functor~$\Phi^!_{\cF^0_{X/S}}\vert_{\Ker(f_*)}$ that follows from Theorem~\ref{thm:k08}.

A similar argument shows that~$\cF$ is a compact generator for~$\Ker(f_*)$ over~$S$.
Indeed, otherwise there is a non-zero object~$\cG \in \Ker(f_*)$
such that~$\Phi^!_\cF(\cG) = 0$.
But then~\eqref{eq:phi-cf-phi-cfi} implies that~$(\Phi^!_{\cF^0_{X/S}}(\cG))_s = 0$ for any~$s \in S$, 
hence~$\Phi^!_{\cF^0_{X/S}}(\cG) = 0$, and hence~$\cG = 0$ by Theorem~\ref{thm:k08}.

Now we prove that~$\cR_\cF$ is a pointwise Clifford algebra.
First, since~$\Phi^!_\cF$ is t-exact, $\cR_\cF \cong \Phi^!_\cF(\cF)$ is a pure algebra.
Furthermore, applying~$f_*$ to the direct sum decomposition~$\cEnd(\cF) \cong \cO_X \oplus \cEnd^0(\cF)$,
where the second summand is the trace-free part, 
and setting~$\cR^0_\cF \coloneqq f_*\cEnd^0(\cF)$, 
we obtain a direct sum decomposition~$\cR_\cF = \cO_S \oplus \cR^0_\cF$, as in~\eqref{eq:cr0}.
Finally, using the argument of the first part of the proof 
together with the isomorphisms~\eqref{eq:asb-csb-fibers} and~\eqref{eq:fs-end-cfi},
we see that
\begin{equation*}
(\cR_\cF)_s \cong 
\Phi^!_\cF(\cF)_s \cong
\Phi^!_{\cF^0_{X/S}}(\cF^0_{X/S})_s \cong
(f_*\cEnd(\cF^0_{X/S}))_s \cong
\Cliff_0(q)_s,
\end{equation*}
so~$\cR_\cF$ is indeed a pointwise Clifford algebra.

Combining all the above we see that~$\cF$ is a compact tilting generator of the category~$\Ker(f_*)$ over~$S$,
hence the functors~$\Phi_\cF$ and~$\Phi^!_\cF$ provide the required $S$-linear t-exact equivalences.
Finally, the isomorphism~$\Phi_\cF(\cR_\cF) \cong \cF$ is obvious from~\eqref{eq:phi-phi}.
\end{proof}

We will also need the following partial converse to Proposition~\ref{prop:cf-cliff}.

\begin{lemma}
\label{lem:s-proj}
Assume~$\cF \in \Ker(f_*)$ is a sheaf such that the functor~$\Phi^!_\cF \colon \Ker(f_*) \to \Db(S)$ is t-exact.
Then~$\cF$ is locally free.
\end{lemma}

\begin{proof}
If~$\cF$ is not locally free, there is a geometric point~$x \in X$ 
such that~$\cRHom(\cF, \cO_x) \not\in \bD(X)^{\le 0}$.
Let~$i_x \colon \Spec(\kk(x)) \to X$ be the inclusion of~$x$.
Then
\begin{equation*}
\cRHom(\cF, \cO_x) \cong i_{x*}\cRHom(i_x^*\cF, \cO_{\kk(x)})
\qquad\text{and}\qquad 
\Phi^!_\cF(\cO_x) 
\cong f_*i_{x*}\cRHom(i_x^*\cF, \cO_{\kk(x)})
\end{equation*}
and since~$i_x$ and~$f \circ i_x$ are closed embeddings, 
the functors~$i_{x*}$ and~$f_* \circ i_{x*}$ are t-exact and conservative,
so it follows that~$\Phi^!_\cF(\cO_x) \not\in \bD(S)^{\le 0}$.
We show below that this leads to a contradiction.

Let~$s = f(x)$, so that~$x \in X_s$.
Since~$X_s$ is a Gorenstein curve, Serre duality implies that
\begin{equation*}
\Ext^\bullet(\cO_x, \omega_{X_s}) \cong \Ext^{1-\bullet}(\cO_{X_s}, \cO_x)^\vee = \kk(x)[-1],
\qquad 
\Ext^\bullet(\cO_{X_s}, \omega_{X_s}) \cong \Ext^{1-\bullet}(\cO_{X_s}, \cO_{X_s})^\vee = \kk(x)[-1],
\end{equation*}
and the 
pairing~$\Ext^\bullet(\cO_{X_s}, \cO_x) \otimes \Ext^\bullet(\cO_x, \omega_{X_s}) \to \Ext^\bullet(\cO_{X_s}, \omega_{X_s})$ 
is perfect, hence there is a unique extension
\begin{equation}
\label{eq:ckx}
0 \to \omega_{X_s} \to \cK_x \to \cO_x \to 0
\end{equation}
such that~$\cK_x \in \Ker(f_*) \cap \Coh(X_s)$.
Furthermore, Serre duality implies
\begin{equation*}
\Phi^!_\cF(\omega_{X_s}) \cong 
\RHom(\cF\vert_{X_s}, \omega_{X_s}) \otimes \cO_s \cong
\RHom(\cO_{X_s}, \cF\vert_{X_s}[1])^\vee \otimes \cO_s \cong
\rH^\bullet(X_s, \cF\vert_{X_s}[1])^\vee \otimes \cO_s =
0,
\end{equation*}
hence applying the functor~$\Phi^!_\cF$ to~\eqref{eq:ckx}, 
we obtain~$\Phi^!_\cF(\cO_x) \cong \Phi^!_\cF(\cK_x)$,
and since~$\Phi^!_\cF$ is t-exact on~$\Ker(f_*)$,
we conclude that~$\Phi^!_\cF(\cO_x) \in \Coh(S)$,
in contradiction to~$\Phi^!_\cF(\cO_x) \not\in \bD(S)^{\le 0}$.
\end{proof}

Now we are ready to give the proofs of our main results.

\begin{proof}[Proof of Theorem~\textup{\ref{thm:asb-cb}}]
By Proposition~\ref{prop:cf-cliff} the algebra~$\cR_\cF \coloneqq f_*\cEnd(\cF)$
is a pointwise Clifford algebra, 
hence by Proposition~\ref{prop:pca-cliff} 
there is a quadratic form~$q_\cF \colon \det(\cR_\cF^0) \to \Sym^2(\cR_\cF^0)$ 
such that~$\cR_\cF \cong \Cliff_0(q_\cF)$.
Let~$X_\cF/S$ be the corresponding conic bundle.
Using the isomorphism~$\cR_\cF \cong \Cliff_0(q_\cF)$ and applying Proposition~\ref{prop:cf-cliff}, 
we obtain a chain of $S$-linear t-exact Fourier--Mukai equivalences of triangulated categories
\begin{equation*}
\Db(S,\Cliff_0(q)) \xrightarrow[{\quad\raisebox{0.5ex}[0ex][0ex]{$\sim$}\quad}]{\ \Phi_{\cF^0_{X/S}}\ }
\Ker(X/S) \xrightarrow[{\quad\raisebox{0.5ex}[0ex][0ex]{$\sim$}\quad}]{\ \ \Phi^!_{\cF_{\vphantom{X/S}}}\ \ }
\Db(S, \cR_\cF) \xrightarrow[{\quad\raisebox{0.5ex}[0ex][0ex]{$\sim$}\quad}]{}
\Db(S,\Cliff_0(q_\cF)) \xrightarrow[{\quad\raisebox{0.5ex}[0ex][0ex]{$\sim$}\quad}]{\ \Phi_{\cF^0_{X_\cF/S}}\ }
\Ker(X_\cF/S).
\end{equation*}

The composition of the last three equivalences proves~\ref{it:ker-derived}.
Moreover, we have
\begin{equation*}
\Phi_\cF^!(\cF) \cong \cR_\cF \cong \Cliff_0(q_\cF)
\qquad\text{and}\qquad
\Phi_{\cF^0_{X_\cF/S}}(\Cliff_0(q_\cF)) \cong \cF^0_{X_\cF/S},
\end{equation*}
by Proposition~\ref{prop:cf-cliff} and Corollary~\ref{cor:cfi-fs},
hence this equivalence takes~$\cF$ to~$\cF^0_{X_\cF/S}$, as required.

On the other hand, the first three equivalences compose to an $S$-linear t-exact equivalence
\begin{equation*}
\Db(S,\Cliff_0(q)) \simeq \Db(S,\Cliff_0(q_\cF)).
\end{equation*}
Therefore, we have an $S$-linear exact equivalence of the abelian categories 
of~$\Cliff_0(q)$-modules and~$\Cliff_0(q_\cF)$-modules on~$S$
and therefore an $S$-linear Morita equivalence of the algebras~$\Cliff_0(q)$ and~$\Cliff_0(q_\cF)$;
this proves~\ref{it:cliff-morita}.

Now we prove the converse part of the theorem.
So, let~$X'/S$ be a conic bundle defined by a quadratic form~$q'$  
and let~$\Psi \colon \Ker(X'/S) \to \Ker(X/S)$ be an $S$-linear t-exact equivalence.
Set
\begin{equation*}
\cF' \coloneqq \Psi(\cF^0_{X'/S}) \in \Ker(X/S).
\end{equation*}
Since~$\Psi$ is t-exact, $\cF'$ is a pure sheaf.
Furthermore, we have an isomorphism of functors
\begin{equation*}
\Phi^!_{\cF^0_{X'/S}}(-) =
f'_*\cRHom(\cF^0_{X'/S}, -) \cong
f_*\cRHom(\Psi(\cF^0_{X'/S}), \Psi(-)) \cong
f_*\cRHom(\cF', \Psi(-)) \cong
\Phi^!_{\cF'}(\Psi(-)).
\end{equation*}
Since~$\Phi^!_{\cF^0_{X'/S}}(-)$ is t-exact and~$\Psi$ is a t-exact equivalence, $\Phi^!_{\cF'}$ is t-exact,
hence Lemma~\ref{lem:s-proj} proves that~$\cF'$ is locally free.
Moreover, since
\begin{equation*}
\cR' \coloneqq 
f_*\cRHom(\cF',\cF') \cong 
f_*\cRHom(\Psi(\cF^0_{X'/S}), \Psi(\cF^0_{X'/S})) \cong 
f'_*\cRHom(\cF^0_{X'/S}, \cF^0_{X'/S}) \cong 
\Cliff_0(q')
\end{equation*}
is a locally free algebra of rank~$4$, 
it follows that the Euler characteristic of the bundle~$\cEnd(\cF'\vert_{X_s})$ on~$X_s$ is~4,
hence the rank of~$\cF'$ is~$2$.
Finally, since~$\Psi$ is $S$-linear 
and~$\cF^0_{X'/S}$ compactly generates~$\Ker(X'/S)$ over~$S$ (Corollary~\ref{cor:cfi-fs}),
the bundle~$\cF'$ compactly generates~$\Ker(X/S)$ over~$S$.
Applying Corollary~\ref{cor:gen-asb}, we conclude that~$\cF'$ is an abstract spinor bundle.
Finally, the isomorphism of algebras~$\cR' \cong \Cliff_0(q')$ observed above 
in combination with Proposition~\ref{prop:pca-cliff} shows that~$X' \cong X_{\cF'}$.
\end{proof}

\begin{remark}
An abstract spinor bundle~$\cF'$ in the proof of the converse part of Theorem~\ref{thm:asb-cb} 
is defined up to an $S$-linear t-exact autoequivalence of~$\Ker(X/S)$.
For instance, in the case where~$X' = X$ (and~$f' = f$), 
one can take~$\cF' = \cF^i_{X/S}$ for any~$i \in \ZZ$ or any twist of these sheaves.
\end{remark}

\begin{proof}[Proof of Corollary~\textup{\ref{cor:he-sm}}]
If~$X'/S$ is a conic bundle hyperbolic equivalent to~$X/S$ 
then its even Clifford algebra~$\Cliff_0(q')$ is Morita equivalent to~$\Cliff_0(q)$ 
(see~\cite[Proposition~1.1(3)]{K24}),
hence the converse part of Theorem~\ref{thm:asb-cb} applies
and we conclude that~$X'/S$ is a spinor modification of~$X/S$.
\end{proof}

\subsection{Corollaries}

We briefly discuss some consequences of Theorem~\ref{thm:asb-cb}.

\begin{corollary}
\label{cor:spinor-modification}
The property of being a spinor modification is an equivalence relation,
i.e., if~$X'/S$ is a spinor modification of~$X/S$ then~$X/S$ is a spinor modification of~$X'/S$,
and if additionally~$X''/S$ is a spinor modification of~$X'/S$
then~$X''/S$ is a spinor modification of~$X/S$.

Moreover, spinor modifications are compatible with base change, 
i.e., if~$X'/S$ is a spinor modification of~$X/S$ and~$T \to S$ is a morphism of schemes 
then~$X'_T/T$ is a spinor modification of~$X_T/T$.
\end{corollary}

\begin{proof}
All claims are obvious because $S$-linear Morita equivalence of even Clifford algebras 
is an equivalence relation, and it is preserved under base changes.
\end{proof}

The following result supports Conjecture~\ref{conj:asb-he} (cf.~\cite[Proposition~1.1(5)]{K24}).

\begin{lemma}
\label{lem:sm-bir}
If~$X'/S$ is a spinor modification of~$X/S$ and the general fiber of~$X/S$ is smooth
then~$X'$ is birational to~$X$ over~$S$.
\end{lemma}

\begin{proof}
By Theorem~\ref{thm:asb-cb} the even Clifford algebras~$\Cliff_0(q)$ and~$\Cliff_0(q')$ are $S$-linear Morita equivalent.
After base change to the function field~$\kk(S)$ of~$S$, 
we obtain a Morita equivalence~$\Cliff_0(q_{\kk(S)}) \sim \Cliff_0(q'_{\kk(S)})$ of~$\kk(S)$-algebras.
If one of these algebras is Morita trivial, then so is the other, 
and then both are isomorphic to the \mbox{$2 \times 2$} matrix algebra over~$\kk(S)$.
Otherwise, since the algebras are 4-dimensional, by the Wedderburn--Artin theorem 
both are division algebras, hence the Morita equivalence 
implies an isomorphism~$\Cliff_0(q_{\kk(S)}) \cong \Cliff_0(q'_{\kk(S)})$.
Finally, Proposition~\ref{prop:pca-cliff} gives an isomorphism~$X_{\kk(S)} \cong X'_{\kk(S)}$ of conics, 
and we conclude that the conic bundles~$X/S$ and~$X'/S$ are birational over~$S$.
\end{proof}

Besides the hyperbolic equivalence and spinor modification equivalence,
there is yet another equivalence relation for conic bundles.
We say that conic bundles~$X/S$ and~$X'/S$ {\sf have equivalent discriminant data} 
if the discriminant divisors of~$X/S$ and~$X'/S$ coincide, i.e., $\Delta_{X/S} = \Delta_{X'/S}$ in~$S$,
and the double coverings~$\tilde\Delta_{X/S} \to \Delta_{X/S}$ and~$\tilde\Delta_{X'/S} \to \Delta_{X'/S}$ 
obtained from the Stein factorizations of the normalizations 
of~$X \times_S \Delta_{X/S} \to \Delta_{X/S}$ and~$X' \times_S \Delta_{X'/S} \to \Delta_{X'/S}$ are isomorphic.

In the case where~$S$ is a smooth rational surface, 
the argument of~\cite[Lemma~3.2]{BB} relying on the Artin--Mumford exact sequence
shows that an equivalence of discriminant data for conic bundles~$X/S$ and~$X'/S$
implies a Morita equivalence~$\Cliff_0(q_{\kk(S)}) \sim \Cliff_0(q'_{\kk(S)})$ 
of the corresponding even Clifford algebras over~$\kk(S)$,
hence, by the argument of Lemma~\ref{lem:sm-bir}, a birational equivalence of~$X/S$ and~$X'/S$.

Note however, that this \emph{does not imply} that the conic bundles are isomorphic,
so the claim of~\cite[Lemma~3.2]{BB} is incorrect 
(though most of the results of~\cite{BB} are correct, 
since they are only concerned with birational properties of conic bundles).

To finish this section, we prove the following useful result.

\begin{proposition}
\label{prop:xprime-smooth}
Let~$X'/S$ be a spinor modification of~$X/S$.
\begin{renumerate}
\item
\label{it:reg}
If~$S$ is regular and~$X$ is regular then~$X'$ is regular.
\item
If~$S$ is smooth over~$\kk$ and~$X$ is smooth over~$\kk$ then~$X'$ is smooth over~$\kk$.
\label{it:sm}
\end{renumerate}
\end{proposition}

\begin{proof}
\ref{it:reg}
We will use the homological criterion for regularity~\cite[Theorem~3.27]{O16}:
a scheme~$X$ is regular if and only if the category~$\Dp(X)$ is regular (i.e., has a strong generator).
We will also use the following result:
a triangulated category with a semiorthogonal decomposition is regular 
if and only if its components are regular~\cite[Propositions~3.20 and~3.22]{O16}.

Since~$S$ is regular, the category~$\Dp(S)$ is regular.
Further, since~$X$ is regular, the category~$\Dp(X)$ is regular, 
and using~\eqref{eq:ker-perf} we deduce that~$\Ker^\perf(X/S)$ is regular.
Since~$\Ker(X/S) \simeq \Ker(X'/S)$ by Theorem~\ref{thm:asb-cb}
and the equivalence is given by a Fourier--Mukai functor, we have~$\Ker^\perf(X/S) \simeq \Ker^\perf(X'/S)$,
hence~$\Ker^\perf(X'/S)$ is regular.
Thus, $\Dp(S)$ and~$\Ker^\perf(X'/S)$ are both regular, hence~$\Dp(X')$ is regular, 
and therefore~$X'$ is regular.

\ref{it:sm}
Let~$\bar\kk$ be an algebraic closure of~$\kk$.
If~$S$ and~$X$ are smooth over~$\kk$, the schemes~$S_{\bar\kk}$ and~$X_{\bar\kk}$ are regular.
Furthermore, $X'_{\bar\kk}$ is a spinor modification of~$X_{\bar\kk}$ by Corollary~\ref{cor:spinor-modification}
hence~$X'_{\bar\kk}$ is regular by~\ref{it:reg}, and therefore~$X'$ is smooth over~$\kk$.
\end{proof}


\section{Almost Fano threefolds with a conic bundle structure}
\label{sec:af3-cb}

In this section we apply the technique of spinor modifications
to describe the structure of the kernel categories for conic bundles related to nonfactorial 1-nodal Fano threefolds.
We also use the obtained descriptions to construct a categorical absorption of singularities
for the corresponding nonfactorial 1-nodal Fano threefolds, see~\cite{KS22,KS23}.

In this section we work over an algebraically closed field~$\kk$ of characteristic~$0$.

\subsection{The conic bundles}

Recall from~\cite[Table~2]{KP23} that there are~4 deformation types 
of nonfactorial 1-nodal prime Fano threefolds~$X$ 
with a small resolution~$\pi \colon Y \to X$ 
such that~$Y$ has a structure of a conic bundle over~$\P^2$ which is not a ~$\P^1$-bundle;
these are types~\typeo{12nb}, \typeo{10na}, \typeo{8nb}, and~\typeo{5n}.
We concentrate on the first three types,
and type~\typeo{5n} will be discussed briefly in~\S\ref{ss:5n}.

By~\cite[Proposition~6.5 and Remark~6.6]{KP23} for each of the types~\typeo{12nb}, \typeo{10na}, \typeo{8nb} 
the conic bundle~$Y/\P^2$ is defined by a quadratic form~$q \colon \cO_{\P^2}(k) \to \Sym^2\cE^\vee$, where
\begin{equation}
\label{eq:k}
k = 
\begin{cases}
\hphantom{-}3, & \text{for type~\typeo{12nb},}\\
\hphantom{-}2, & \text{for type~\typeo{10na},}\\
\hphantom{-}1, & \text{for type~\typeo{8nb},}
\end{cases}
\end{equation} 
and~$\cE$ is a vector bundle of rank~3 on~$\P^2$
whose dual fits into the following exact sequence
\begin{equation}
\label{eq:cev}
0 \longrightarrow 
\cO_{\P^2}^{\oplus (k - 1)} \longrightarrow 
\cO_{\P^2}(1)^{\oplus (k + 2)} \longrightarrow 
\cE^\vee \longrightarrow \cO_L 
\longrightarrow 0,
\end{equation} 
where~$L \subset \P^2$ is a line.
Note that the discriminant divisor of~$Y/\P^2$ has degree~$6 - k \in \{3,4,5\}$;
we will show below that it does not contain the line~$L$.

The semiorthogonal decomposition~\eqref{eq:dbx} takes for the conic bundle~$f \colon Y \to \P^2$ the form
\begin{equation}
\label{eq:sod-y-cb}
\Db(Y) = \langle \Ker(f_*), f^*(\Db(\P^2)) \rangle.
\end{equation}
We will construct an exceptional object in~$\Ker(f_*)$ and describe its orthogonal complement.
Note that the canonical spinor bundles~$\cF^i_{Y/\P^2}$ are not exceptional; in fact, 
\begin{equation*}
\Ext^\bullet(\cF^i_{Y/\P^2}, \cF^i_{Y/\P^2}) \cong
\rH^\bullet(\P^2, \Cliff_0(q)) \cong 
\rH^\bullet(\P^2, \cO_{\P^2} \oplus \wedge^2\cE(k)) \cong
\kk \oplus \rH^\bullet(\P^2, \cE^\vee(-3))
\end{equation*}
by~\eqref{eq:fs-end-cfi} and~\eqref{eq:cl0},
and taking~\eqref{eq:cev} into account, 
we compute~$\Ext^\bullet(\cF^i_{Y/\P^2}, \cF^i_{Y/\P^2}) \cong \kk \oplus \kk^{\oplus (k + 1)}[-1]$.

We denote by~$h$ the line class of~$\P^2$ (as well as its pullback to~$\P_{\P^2}(\cE)$ and~$Y$)
and by~$H$ the relative hyperplane class of~$\P_{\P^2}(\cE)$ (as well as its restriction to~$Y$).
Note that
\begin{equation}
\label{eq:ky}
\pi^*K_X = K_Y = -H
\qquad\text{and}\qquad 
K_{Y/\P^2} = 3h - H.
\end{equation} 
The dual of the epimorphism~\mbox{$\cE^\vee \twoheadrightarrow \cO_L$} gives an embedding~$\cO_L \hookrightarrow \cE\vert_L$,
hence a section
\begin{equation}
\label{eq:curve-c}
C \subset \P_L(\cE\vert_L) \subset \P_{\P^2}(\cE)
\end{equation}
of the projection~$\P_L(\cE\vert_L) \to L$ such that
\begin{equation}
\label{eq:h-c}
H \cdot C = 0
\qquad\text{and}\qquad
h \cdot C = 1.
\end{equation}
In particular, \eqref{eq:h-c} shows that~$C$ is $K$-trivial, 
and since~$\pi \colon Y \to X$ is a small resolution of a Fano variety, 
we see that~$C$ is the exceptional curve of~$\pi$, and since~$X$ is 1-nodal, we have
\begin{equation}
\label{eq:cn-c-y}
\cN_{C/Y} \cong \cO_C(-1)^{\oplus 2}.
\end{equation}

Note that since~$\g(X) \ge 7$, it follows from~\cite[Proposition~5.2]{KP23}
that there are only finitely many anticanonical lines on~$X$ through the node,
hence there are only finitely many half-fibers of the conic bundle~$Y/\P^2$ intersecting the curve~$C$.
Therefore, the line~$L$ is not contained in the discriminant of~$Y/\P^2$.

Further, since~$X$ is a prime Fano threefold, i.e., $\Pic(X) = \ZZ \cdot K_X$, 
it follows from~\eqref{eq:h-c} that
\begin{equation}
\label{eq:picy}
\Pic(Y) = \ZZ \cdot H \oplus \ZZ \cdot h
\end{equation}
(alternatively, this follows from~\cite[Proposition~3.3]{KP23}).

We denote by~$\cI_{C,Y}$ the ideal of~$C$ on~$Y$.
Our main observation is the following

\begin{lemma}
\label{lem:pf-cf1}
There is a unique non-split extension
\begin{equation}
\label{eq:def-cf1}
0 \to \cO_Y(2h - H) \to \cF \to \cI_{C,Y} \to 0
\end{equation}
and the sheaf~$\cF$ defined by~\eqref{eq:def-cf1} is an abstract spinor bundle on~$Y/\P^2$ with~$\rc_1(\cF) = 2h - H$.

In particular, we have~$\rH^\bullet(Y, \cF(th)) = 0$ for all~$t \in \ZZ$.
\end{lemma}

\begin{proof}
Since~$C \subset Y_L \coloneqq f^{-1}(L)$ and the ideal of~$Y_L$ is isomorphic to~$\cO_Y(-h)$, 
there is an exact sequence
\begin{equation*}
0 \to \cO_Y(-h) \to \cI_{C,Y} \to \cI_{C,Y_L} \to 0,
\end{equation*}
where~$\cI_{C,Y_L}$ is the ideal of~$C$ on~$Y_L$,
and since~$f\vert_C \colon C \to L$ is an isomorphism, we have~$f_*(\cI_{C,Y_L}) = 0$.
Hence, by adjunction, we have~$\Ext^\bullet(\cO_Y(2h), \cI_{C,Y_L}) = 0$, 
and then, using~\eqref{eq:ky} and Serre duality, we deduce the vanishing~$\Ext^\bullet(\cI_{C,Y_L}, \cO_Y(2h - H)) = 0$.
Now applying the functor~$\Ext^\bullet(-, \cO_Y(2h - H))$ to the above exact sequence, we conclude that
\begin{equation*}
\Ext^\bullet(\cI_{C,Y}, \cO_Y(2h - H)) \cong
\Ext^\bullet(\cO_Y(-h), \cO_Y(2h - H)) \cong
\rH^\bullet(Y, \cO_Y(3h - H)) \cong \kk[-1].
\end{equation*}
Therefore, we have a commutative diagram of unique non-split extensions
\begin{equation*}
\xymatrix@R=4ex{
0 \ar[r] & 
\cO_Y(2h - H) \ar[r] \ar@{=}[d] & 
\cF^0(-h) \ar[r] \ar[d] & 
\cO_Y(-h) \ar[r] \ar[d] & 
0
\\
0 \ar[r] & 
\cO_Y(2h - H) \ar[r] & 
\cF \ar[r] & 
\cI_{C,Y} \ar[r] & 
0,
}
\end{equation*}
where the top row is obtained from~\eqref{eq:csb} by a twist,
and the middle arrow extends to an exact sequence
\begin{equation*}
0 \to \cF^0(-h) \to \cF \to \cI_{C,Y_L} \to 0.
\end{equation*}
As we noticed above, $f_*(\cI_{C,Y_L}) = 0$, hence~$f_*(\cF) \cong f_*(\cF^0(-h)) = 0$,
because~$\cF^0$ is an abstract spinor bundle by Lemma~\ref{lem:cf0-cf1}.
Further, the bottom row implies that~$\rank(\cF) = 2$ and~$\rc_1(\cF) = 2h - H = K_{Y/\P^2} - h$,
hence to prove that~$\cF$ is an abstract spinor bundle we only need to check that it is locally free.
For this we show that~$\cF$ is the vector bundle of rank~2 obtained from~$C \subset Y$ by Serre's construction.

Indeed, using~\eqref{eq:h-c} and~\eqref{eq:cn-c-y}, we find~$\cO_Y(2h - H)\vert_C \cong \cO_C(2) \cong \det\cN^\vee_{C/Y}$,
while
\begin{equation*}
\rH^\bullet(Y, \cO_Y(2h - H)) = 
\rH^\bullet(Y, \omega_{Y/\P^2}(-h)) = 
\rH^{\bullet-1}(\P^2, \cO_{\P^2}(-1)) = 0,
\end{equation*}
hence Serre's construction is well defined and produces a vector bundle 
that can be represented as the unique non-split extension~\eqref{eq:def-cf1};
in particular, it coincides with~$\cF$ defined above.

The last statement follows from~$\rH^\bullet(Y, \cF(th)) = \rH^\bullet(\P^2, f_*\cF \otimes \cO_{\P^2}(t))$,
because~$\cF \in \Ker(f_*)$.
\end{proof}

Restricting~\eqref{eq:def-cf1} to~$C$ we obtain 
an epimorphism~$\cF\vert_C \twoheadrightarrow \cI_{C,Y} / \cI_{C,Y}^2 \cong \cN^\vee_{C/Y}$; therefore~\eqref{eq:cn-c-y} implies
\begin{equation}
\label{eq:cf1-c}
\cF\vert_C \cong \cO_C(1)^{\oplus 2}.
\end{equation}
Later we will show that~$\cF$ is exceptional, see Corollary~\ref{cor:cf-exc}.

\subsection{The spinor modification of~$Y$}

In this subsection we describe the spinor modification~$Y_\cF/\P^2$ of the conic bundle~$Y/\P^2$ 
with respect to the abstract spinor bundle~$\cF$ constructed in Lemma~\ref{lem:pf-cf1}.
Before giving an explicit description of~$Y_\cF$, we observe the following useful property.

\begin{lemma}
\label{lem:no-sections}
If\/~$Y_\cF/\P^2$ is the $\cF$-modification of~$Y/\P^2$, the map~$Y_\cF \to \P^2$ has no rational sections.
\end{lemma}

\begin{proof}
If~$Y_\cF \to \P^2$ has a rational section then by Lemma~\ref{lem:sm-bir} the conic bundle~$Y \to \P^2$ 
also has a rational section~$\P^2 \dashrightarrow Y$.
Thus, $\Pic(Y)$ contains a divisor class which has intersection index~$1$ 
with the numerical class~$\Upsilon$ of fibers of~$Y \to \P^2$.
This, however, contradicts~\eqref{eq:picy}, because~$H \cdot \Upsilon = 2$ and~$h \cdot \Upsilon = 0$. 
\end{proof}

Recall that~$\cEnd^0(\cF)$ denotes the trace-free part of the endomorphism bundle~$\cEnd(\cF)$.

\begin{proposition}
\label{prop:good-cb0}
Let~$Y$ be a conic bundle of type~\typeo{12nb}, \typeo{10na}, or~\typeo{8nb}.
If~$k$ is defined by~\eqref{eq:k} then
\begin{equation}
\label{eq:fs-cend-cf}
f_*\cEnd^0(\cF) \cong \cO_{\P^2}(-1)^{\oplus k} \oplus \cO_{\P^2}(-2)^{\oplus (3 - k)}.
\end{equation}
In particular, the conic bundle~$Y_\cF \to \P^2$ 
corresponds to a quadratic form~$q_\cF \colon \cL_\cF \to \Sym^2\cE_\cF^\vee$, where
\begin{equation}
\label{eq:yp}
\left\{
\begin{aligned}
& \cL_\cF \cong \cO_{\P^2}(-1), 
&& \cE_\cF \cong \cO_{\P^2} \oplus \cO_{\P^2} \oplus \cO_{\P^2}, 
&& \text{for type~\typeo{12nb}},\\
& \cL_\cF \cong \cO_{\P^2}, 
&& \cE_\cF \cong \cO_{\P^2}(-1) \oplus \cO_{\P^2}(-1) \oplus \cO_{\P^2}, 
&& \text{for type~\typeo{10na}},\\
& \cL_\cF \cong \cO_{\P^2}(-1), 
&& \cE_\cF \cong \cO_{\P^2}(-1) \oplus \cO_{\P^2} \oplus \cO_{\P^2}, 
&& \text{for type~\typeo{8nb}},\\
\end{aligned}
\right.
\end{equation}
and the total space~$Y_\cF \subset \P_{\P^2}(\cE_\cF)$ of this conic bundle is smooth.
\end{proposition}

Note that these conic bundles are examples of conic bundles 
of type~$\bF^2_3$, $\bF^2_4$, and~$\bF^{2-}_5$ from~\cite[\S8]{CI}.

\begin{proof}
We start by computing~$f_*(\cF^\vee) \cong f_*(\cF(H-2h))$.
For this we twist~\eqref{eq:def-cf1} by~$\cO_Y(H-2h)$ and pushing it forward to~$\P^2$, we obtain a distinguished triangle
\begin{equation*}
\cO_{\P^2} \to f_*(\cF(H-2h)) \to f_*(\cI_{C,Y}(H-2h)).
\end{equation*}
Similarly, twisting the exact sequence~$0 \to \cI_{C,Y} \to \cO_Y \to \cO_C \to 0$ and pushing it forward, we obtain
\begin{equation*}
f_*(\cI_{C,Y}(H-2h)) \to \cE^\vee(-2) \to \cO_L(-2).
\end{equation*}
By definition~\eqref{eq:curve-c} of the curve~$C \subset Y$, the last arrow here 
coincides with a twist of the last arrow in~\eqref{eq:cev}.
In particular, it is surjective, and we obtain an exact sequence
\begin{equation}
\label{eq:fs-fvee}
0 \to \cO_{\P^2} \to f_*(\cF(H-2h)) \to \cK \to 0,
\end{equation}
where~$\cK$ is the vector bundle of rank~3 on~$\P^2$ that fits into an exact sequence
obtained from~\eqref{eq:cev} by truncation and twist:
\begin{equation}
\label{eq:ck-seq}
0 \to \cO_{\P^2}(-2)^{\oplus (k - 1)} \to \cO_{\P^2}(-1)^{\oplus (k + 2)} \to \cK \to 0.
\end{equation} 
The last sequence immediately implies the following cohomology vanishing
\begin{equation}
\label{eq:hb-ck}
\rH^\bullet(\P^2,\cK) = 0.
\end{equation} 

Now we tensor~\eqref{eq:def-cf1} by~$\cF^\vee \cong \cF(H-2h)$ and obtain an exact sequence
\begin{equation*}
0 \to \cF \to \cEnd(\cF) \to \cF(H - 2h) \to \cF(H-2h) \otimes \cO_C \to 0.
\end{equation*}
By Lemma~\ref{lem:pf-cf1} we have~$f_*\cF = 0$, hence we obtain an exact sequence
\begin{equation*}
0 \to f_*\cEnd(\cF) \to f_*(\cF(H - 2h)) \to f_*(\cF(H-2h) \otimes \cO_C) \to 0.
\end{equation*}
Since~$f\vert_C \colon C \to L$ is an isomorphism, it follows from~\eqref{eq:h-c} and~\eqref{eq:cf1-c} 
that the last term is~$\cO_L(-1)^{\oplus 2}$.
On the other hand, the first term splits as~$f_*\cEnd(\cF) \cong \cO_{\P^2} \oplus f_*\cEnd^0(\cF)$
and the second term is the extension~\eqref{eq:fs-fvee}.
Since~$\rH^0(\P^2, \cK) = 0$ by~\eqref{eq:hb-ck}, we conclude that the first summand~$\cO_{\P^2}$ of~$f_*\cEnd(\cF)$ 
is mapped isomorphically onto the first term of~\eqref{eq:fs-fvee},
hence we obtain an exact sequence
\begin{equation}
\label{eq:end-cf-ck}
0 \to f_*\cEnd^0(\cF) \to \cK \to \cO_L(-1)^{\oplus 2} \to 0.
\end{equation}
Clearly, the composition~$\cO_{\P^2}(-1)^{\oplus (k + 2)} \to \cK \to \cO_L(-1)^{\oplus 2}$ 
of the second arrows in~\eqref{eq:ck-seq} and~\eqref{eq:end-cf-ck} is surjective 
and its kernel is isomorphic to~$\cO_{\P^2}(-1)^{\oplus k} \oplus \cO_{\P^2}(-2)^{\oplus 2}$,
hence we have an exact sequence
\begin{equation}
\label{eq:end-cf}
0 \to \cO_{\P^2}(-2)^{\oplus (k - 1)} \to \cO_{\P^2}(-1)^{\oplus k} \oplus \cO_{\P^2}(-2)^{\oplus 2} \to f_*\cEnd^0(\cF) \to 0.
\end{equation}
Consider the component~$\cO_{\P^2}(-2)^{\oplus (k - 1)} \to \cO_{\P^2}(-2)^{\oplus 2}$ of the first map.
It is given by a constant matrix, hence it has constant rank~$r \le \min(k - 1,2) = k - 1$ (see~\eqref{eq:k}), 
hence~$f_*\cEnd^0(\cF)$ is isomorphic to the direct sum of~$\cO_{\P^2}(-2)^{\oplus (2 - r)}$
and the cokernel of a morphism~$\cO_{\P^2}(-2)^{\oplus (k - 1 - r)} \to \cO_{\P^2}(-1)^{\oplus k}$.
Since~$f_*\cEnd^0(\cF)$ is locally free, we conclude that~$k - 1 - r \le \max(k - 2,0)$.

First, consider the case where~$k - 1 - r = 0$.
Then~\eqref{eq:fs-cend-cf} follows immediately from the above arguments.
Moreover, in this case Theorem~\ref{thm:asb-cb} implies that
\begin{equation*}
\cL_\cF \cong \det(f_*\cEnd^0(\cF)) \cong \cO_{\P^2}(k - 6)
\qquad\text{and}\qquad 
\cE_\cF \cong (f_*\cEnd^0(\cF))^\vee \cong \cO_{\P^2}(1)^{\oplus k} \oplus \cO_{\P^2}(2)^{\oplus (3 - k)}.
\end{equation*}
Twisting by~$\cO_{\P^2}(-1)$, if~$k = 3$ and by~$\cO_{\P^2}(-2)$, if~$k \in \{1,2\}$, we obtain~\eqref{eq:yp}.
Finally, the smoothness of~$Y_\cF$ follows from the smoothness of~$Y$ by Proposition~\ref{prop:xprime-smooth}.

It remains to exclude the case where~$k - 2 \ge k - 1 - r > 0$.
Since~$k \le 3$ by~\eqref{eq:k}, it follows that~$k = 3$ and~$r = 1$.
In this case~$f_*\cEnd^0(\cF)$ is the direct sum of~$\cO_{\P^2}(-2)$ 
and the cokernel of a morphism~$\cO_{\P^2}(-2) \to \cO_{\P^2}(-1)^{\oplus 3}$.
Since~$f_*\cEnd^0(\cF)$ must be locally free, the morphism must be isomorphic to a twist of the tautological embedding,
and using Theorem~\ref{thm:asb-cb} we conclude that
\begin{equation*}
\cL_\cF \cong \det(f_*\cEnd^0(\cF)) \cong \cO_{\P^2}(-3)
\qquad\text{and}\qquad 
\cE_\cF \cong (f_*\cEnd^0(\cF))^\vee \cong \cO_{\P^2}(2) \oplus \Omega_{\P^2}(2).
\end{equation*}
But then the section of~$\P_{\P^2}(\cE_\cF)$ corresponding to the summand~$\cO_{\P^2}(2)$ is contained in~$Y_\cF$,
hence the morphism~$Y_\cF \to \P^2$ has a section in contradiction to Lemma~\ref{lem:no-sections}.
This contradiction shows that this case is impossible, and completes the proof of the proposition.
\end{proof}

\begin{corollary}
\label{cor:cf-exc}
The abstract spinor bundle~$\cF$ constructed in Lemma~\textup{\ref{lem:pf-cf1}} is exceptional 
and~$\rH^\bullet(Y, \cF^\vee) = \kk$.
\end{corollary}

\begin{proof}
Indeed,
$\Ext^\bullet(\cF, \cF) \cong 
\rH^\bullet(\P^2, f_*\cEnd(\cF)) \cong 
\rH^\bullet(\P^2, \cO_{\P^2} \oplus f_*\cEnd^0(\cF)) \cong 
\kk$,
where we use~\eqref{eq:fs-cend-cf}.
Similarly, a combination of~\eqref{eq:fs-fvee} and~\eqref{eq:hb-ck} 
implies that~$\rH^\bullet(Y, \cF^\vee) = \rH^\bullet(Y, \cF(H - 2h)) = \kk$.
\end{proof}

Using the description of~$\cE_\cF$ and~$\cL_\cF$ in Proposition~\ref{prop:good-cb0} we can also interpret~$Y_\cF$ geometrically.

\begin{corollary}
\label{cor:yprime-12-10-8}
Let~$Y_\cF/\P^2$ be the $\cF$-modification of~$Y/\P^2$.
Then
\begin{itemize}[wide]
\item 
$Y_\cF \subset \P^2 \times \P^2 \xrightarrow{\ \pr_2\ } \P^2$ is a hypersurface of bidegree~$(2,1)$ 
for type~\typeo{12nb};
\item 
$Y_\cF \to \P^1 \times \P^2 \xrightarrow{\ \pr_2\ } \P^2$ is a double covering with branch divisor of bidegree~$(2,2)$
for type~\typeo{10na}.
\item 
$Y_\cF = \Bl_\ell(\bar{Y}) \xrightarrow{\ \ \ \ \ \ } \P^2$ is the blowup 
of a smooth cubic threefold~$\bar{Y}$ in a line~$\ell \subset \bar{Y}$
for type~\typeo{8nb}.
\end{itemize}
In particular, $Y_\cF$ is a smooth Fano threefold 
of type~\typemm{2}{24}, or~\typemm{2}{18}, or~\typemm{2}{11}, respectively.
\end{corollary}

\begin{proof}
The smoothness of~$Y_\cF$ in all cases is explained in Proposition~\ref{prop:good-cb0}.

For type~\typeo{12nb} the description~\eqref{eq:yp} shows 
that~$\P_{\P^2}(\cE_\cF) = \P^2 \times \P^2$ and~$Y_\cF$ is a divisor of bidegree~$(2,1)$.

Similarly, for type~\typeo{10na} the fourfold~$\P_{\P^2}(\cE_\cF) = \P_{\P^2}(\cO_{\P^2}(-1) \oplus \cO_{\P^2}(-1) \oplus \cO_{\P^2})$
is a small resolution of the cone~$\Cone(\P^1 \times \P^2) \subset \P^6$
and~$Y_\cF$ is the preimage of the intersection of this cone with a quadric~$Q \subset \P^6$.
If~$Q$ contains the vertex~$\upsilon$ of the cone, the preimage of~$\upsilon$
is a section of the morphism~$Y_\cF \to \P^2$, in contradiction to Lemma~\ref{lem:no-sections}.
Thus, $\upsilon \not\in Q$, and therefore the linear projection out of~$\upsilon$ 
identifies~$Y_\cF$ with a double covering of~$\P^1 \times \P^2$
ramified over a divisor of bidegree~$(2,2)$.

Finally, for type~\typeo{8nb} 
we have~$\P_{\P^2}(\cE_\cF) = \P_{\P^2}(\cO_{\P^2}(-1) \oplus \cO_{\P^2} \oplus \cO_{\P^2}) \cong \Bl_\ell(\P^4)$,
where~$\ell \subset \P^4$ is a line,
and~$Y_\cF$ is the strict transform of a cubic threefold~$\bar{Y} \subset \P^4$ containing~$\ell$ with multiplicity~1.
Thus, we have~$Y_\cF \cong \Bl_\ell(\bar{Y})$.
Let~$E \subset Y_\cF$ be the exceptional divisor.
Obviously, $E$ is a hypersurface in the exceptional divisor~$\ell \times \P^2$ of~$\Bl_\ell(\P^4)$ 
of relative degree~$1$ over~$\ell$.
If a fiber of the projection~$\ell \times \P^2 \to \ell$ is contained in~$E$, 
it provides a section of the projection~$Y_\cF \to \P^2$, 
in contradiction to Lemma~\ref{lem:no-sections}.
Therefore, $E$ is a $\P^1$-bundle over~$\ell$, 
hence~$\ell \subset \bar{Y}$ is a local complete intersection,
and smoothness of~$Y_\cF$ implies smoothness of~$\bar{Y}$.
\end{proof}

\begin{remark}
Conjecture~\ref{conj:asb-he} predicts that~$Y_\cF/\P^2$ is hyperbolic equivalent to~$Y/\P^2$.
It would be interesting to find the required hyperbolic equivalence.
It is also interesting to find an abstract spinor bundle on~$Y_\cF$ 
such that the corresponding spinor modification of~$Y_\cF$ is~$Y$.
\end{remark}

\subsection{The orthogonal complement of~$\cF$}

Recall that the abstract spinor bundle~$\cF \in \Ker(Y/\P^2)$ constructed in Lemma~\ref{lem:pf-cf1} 
is exceptional by Corollary~\ref{cor:cf-exc}, therefore we have a semiorthogonal decomposition
\begin{equation}
\label{eq:cfperp}
\Ker(Y/\P^2) = \langle \cF^\perp, \cF \rangle.
\end{equation}
In this subsection we 
describe the orthogonal complement~$\cF^\perp \subset \Ker(Y/\P^2)$.

We start with the case of a conic bundle of type~\typeo{12nb}.
We denote by
\begin{equation*}
\Qu_3 = \left( \xymatrix@1{\bullet \ar@/^.9ex/[r] \ar@/_.9ex/[r] \ar[r] & \bullet} \right)
\end{equation*}
{\sf the $3$-Kronecker quiver}, i.e., the quiver with two vertices and three arrows 
and by~$\Db(\Qu_3)$ the bounded derived category of its representations.
The following result can be deduced from the computation of~\cite[Proposition~5.8]{BB};
we provide here an alternative argument.

\begin{proposition}
\label{prop:cby-12}
If~$Y/\P^2$ is a conic bundle of type~\typeo{12nb},
$\cF \in \Ker(Y/\P^2)$ is the exceptional abstract spinor bundle on~$Y$ constructed in Lemma~\textup{\ref{lem:pf-cf1}},
and~$\cF^\perp$ is defined by~\eqref{eq:cfperp} then~$\cF^\perp \simeq \Db(\Qu_3)$. 
\end{proposition}

\begin{proof}
By Theorem~\ref{thm:asb-cb} there is an equivalence~$\Ker(Y/\P^2) \simeq \Ker(Y_\cF/\P^2)$
that takes~$\cF$ to~$\cF^0_{Y_\cF/\P^2}$, hence
\begin{equation*}
\cF^\perp \simeq (\cF^0_{Y_\cF/\P^2})^\perp \subset \Ker(Y_\cF/\P^2).
\end{equation*}
To describe this category we use the following observation.
Recall that~$Y_\cF \subset \P(V_1) \times \P(V_2)$ is a divisor of bidegree~$(2,1)$, 
where~$V_1$ and~$V_2$ are vector spaces of dimension~$3$.
The equation~$q_\cF \in \Sym^2V_1^\vee \otimes V_2^\vee$ of~$Y_\cF$ 
induces a linear embedding~$V_2 \hookrightarrow \Sym^2V_1^\vee$ 
and~$Y_\cF$ by definition coincides with the base change 
of the universal conic~$\cC \subset \P(V_1) \times \P(\Sym^2V_1^\vee)$
along the induced map~$\P(V_2) \to \P(\Sym^2V_1^\vee)$.
Since the intersection in~$\P(V_1)$ of the conics parameterized by~$\P(V_2)$ is empty
(because the morphism~$Y_\cF \to \P(V_2)$ has no sections by Lemma~\ref{lem:no-sections}), 
\cite[Theorem~5.5]{K08} gives a full exceptional collection
\begin{equation*}
\Db(\P(V_2), \Cliff_0(q_\cF)) = 
\Big\langle \Cliff_{-2}(q_\cF), \Cliff_{-1}(q_\cF), \Cliff_0(q_\cF) \Big\rangle.
\end{equation*}
Furthermore, applying the equivalence~$\Phi_{\cF^0_{Y_\cF/\P(V_2)}} \colon 
\Db(\P(V_2), \Cliff_0(q_\cF)) \xrightarrow[{\ \raisebox{0.5ex}[0ex][0ex]{$\sim$}\ }]{} \Ker(q_\cF)$ 
from Theorem~\ref{thm:k08} and using Corollary~\ref{cor:cfi-fs}, we obtain a semiorthogonal decomposition
\begin{equation*}
\Ker(q_\cF) = \Big\langle \cF^{-2}_{Y_\cF/\P(V_2)}, \cF^{-1}_{Y_\cF/\P(V_2)}, \cF^0_{Y_\cF/\P(V_2)} \Big\rangle.
\end{equation*}
This shows that the category~$(\cF^0_{Y_\cF/\P(V_2)})^\perp$ 
is generated by the exceptional pair~$\cF^{-2}_{Y_\cF/\P(V_2)}$, $\cF^{-1}_{Y_\cF/\P(V_2)}$.
Finally, applying Corollary~\ref{cor:cfi-fs} and~\eqref{eq:clifford-all}, we compute
\begin{equation*}
\Ext^\bullet(\cF^{-2}_{Y_\cF/\P(V_2)}, \cF^{-1}_{Y_\cF/\P(V_2)}) \cong
\rH^\bullet(\P(V_2), \Cliff_1(q_\cF)) \cong 
\rH^\bullet(\P(V_2), V_1 \otimes \cO_{\P(V_2)} \oplus \cO_{\P(V_2)}(-1)) \cong
V_1,
\end{equation*}
which shows that the category~$(\cF^0_{Y_\cF/\P(V_2)})^\perp$ is equivalent to~$\Db(\Qu_3)$.
\end{proof}

\begin{remark}
It would be interesting to identify the abstract spinor bundle on~$Y$ equal to the image of~$\cF^{-1}_{Y_\cF/\P^2}$ 
under the equivalence~$\Ker(Y_\cF/\P^2) \simeq \Ker(Y/\P^2)$.
\end{remark}

Next, we consider the conic bundle of type~\typeo{10na}.

\begin{proposition}
\label{prop:cby-10}
If~$Y/\P^2$ is a conic bundle of type~\typeo{10na},
$\cF \in \Ker(Y/\P^2)$ is the exceptional abstract spinor bundle on~$Y$ constructed in Lemma~\textup{\ref{lem:pf-cf1}},
and~$\cF^\perp$ is defined by~\eqref{eq:cfperp} then~$\cF^\perp \simeq \Db(\Gamma_2)$, 
where the right side is the derived category of a curve~$\Gamma_2$ of genus~$2$.
\end{proposition}

\begin{proof}
As before, it is enough to identify the orthogonal~$(\cF^0_{Y_\cF/\P^2})^\perp \subset \Ker(Y_\cF/\P^2)$ 
with~$\Db(\Gamma_2)$.
Recall from Corollary~\ref{cor:yprime-12-10-8} that~$Y_\cF$ is a smooth double covering of~$\P^1 \times \P^2$
with branch divisor of bidegree~$(2,2)$.
We will use the fact that the first projection~$Y_\cF \to \P^1$ is a quadric surface bundle.
More precisely, 
\begin{equation*}
Y_\cF \subset \P_{\P^1}(\cO_{\P^1} \oplus \cO_{\P^1}(-1)^{\oplus 3})
\end{equation*}
corresponds to a quadratic form~$q'' \colon \cO_{\P^1} \to \Sym^2(\cO_{\P^1} \oplus \cO_{\P^1}(1)^{\oplus 3})$.
The discriminant divisor of~$q''$ has degree~$2\rc_1(\cO_{\P^1} \oplus \cO_{\P^1}(1)^{\oplus 3}) = 6$,
and since~$Y_\cF$ is smooth, the discriminant divisor is reduced.

Let~$\Gamma_2 \to \P^1$ be the double covering branched at the discriminant divisor of~$q''$; this is a smooth curve of genus~2.
A combination of~\cite[Theorem~4.2]{K08}, \cite[Proposition~3.13]{K08}, 
and the vanishing of the Brauer group~$\Br(\Gamma_2) = 0$
(recall that in this section the base field is assumed to be algebraically closed)
implies that there is a semiorthogonal decomposition
\begin{equation*}
\Db(Y_\cF) = \langle \Psi(\Db(\Gamma_2)), \cO_{Y_\cF}, \cO_{Y_\cF}(h_1), \cO_{Y_\cF}(h_2), \cO_{Y_\cF}(h_1 + h_2) \rangle,
\end{equation*}
where~$\Psi \colon \Db(\Gamma_2) \to \Db(Y_\cF)$ is a fully faithful embedding, 
while~$h_1$ and~$h_2$ are the pullbacks to~$Y_\cF$ of the hyperplane classes of~$\P^1$ and~$\P^2$, respectively.
Now we apply a sequence of mutations.

First, we mutate~$\cO_{Y_\cF}(h_1 + h_2)$ to the far left.
Since~$K_{Y_\cF} = -h_1 - 2h_2$, we obtain 
\begin{equation*}
\Db(Y_\cF) = \langle \cO_{Y_\cF}(-h_2), \Psi(\Db(\Gamma_2)), \cO_{Y_\cF}, \cO_{Y_\cF}(h_1), \cO_{Y_\cF}(h_2) \rangle.
\end{equation*}

Next, mutating~$\Psi(\Db(\Gamma_2))$ to the left of~$\cO_{Y_\cF}(-h_2)$, we obtain
\begin{equation*}
\Db(Y_\cF) = \langle \Psi'(\Db(\Gamma_2)), \cO_{Y_\cF}(-h_2), \cO_{Y_\cF}, \cO_{Y_\cF}(h_1), \cO_{Y_\cF}(h_2) \rangle,
\end{equation*}
where~$\Psi'$ is the composition of~$\Psi$ with the mutation functor.

Finally, we mutate~$\cO_{Y_\cF}(h_1)$ two steps to the left.
Since~$\cO_{Y_\cF}$ and~$\cO_{Y_\cF}(h_1)$ are pullback from~$\P^1$,
the first mutation is a pullback of the mutation of~$\cO_{\P^1}(1)$ to the left of~$\cO_{\P^1}$, 
hence the result is~$\cO_{Y_\cF}(-h_1)$.
To compute the mutation of~$\cO_{Y_\cF}(-h_1)$ through~$\cO_{Y_\cF}(-h_2)$, we note that~$K_{Y_\cF/\P^2} = h_2 - h_1$, 
hence the mutation is realized by the exact sequence
\begin{equation*}
0 \to \cO_{Y_\cF}(-h_1) \to \cF^0(-h_2) \to \cO_{Y_\cF}(-h_2) \to 0,
\end{equation*}
obtained from~\eqref{eq:csb} by a twist.
Thus, we obtain a semiorthogonal decomposition
\begin{equation*}
\Db(Y_\cF) = \langle \Psi'(\Db(\Gamma_2)), \cF^0(-h_2), \cO_{Y_\cF}(-h_2), \cO_{Y_\cF}, \cO_{Y_\cF}(h_2) \rangle.
\end{equation*}
It follows that~$\Ker(Y_\cF/\P^2) =\langle \Psi'(\Db(\Gamma_2)), \cF^0(-h_2) \rangle$,
hence~$(\cF^0_{Y_\cF/\P^2})^\perp \simeq (\cF^0(-h_2))^\perp \simeq \Db(\Gamma_2)$. 
\end{proof}

\begin{remark}
It is not hard to see that the embedding~$\Psi' \colon \Db(\Gamma_2) \hookrightarrow \Db(Y_\cF)$ 
is given by the universal bundle for the moduli space whose typical member 
is obtained up to twist by Serre's construction applied to a line~$\ell$ 
in the fiber of the first projection~$Y_\cF \to \P^1$,
i.e., to the vector bundle~$\cF_\ell$ defined by the sequence
\begin{equation*}
0 \to \cO_{Y_\cF}(-h_1-h_2) \to \cF_\ell \to \cO_{Y_\cF}(-h_2) \to \cO_\ell(-1) \to 0.
\end{equation*}
As~$\cF_\ell$ are vector bundles of rank~2 in~$\Ker(Y_\cF/\P^2)$ 
with~$\rc_1(\cF_\ell) = -h_1 - 2h_2 = K_{Y_\cF}$,
they are abstract spinor bundles.
It would be interesting to identify the corresponding abstract spinor bundles on~$Y$.
\end{remark}

\begin{remark}
The semiorthogonal decomposition~$\Db(\P^2, \Cliff_0(q_\cF)) = \langle \Db(\Gamma_2), \Cliff_0(q_\cF) \rangle$
proved by the mutation argument of Proposition~\ref{prop:cby-10} also follows from~\cite[Theorem~1.0.3]{Shen}.
\end{remark}

Finally, we consider the conic bundle of type~\typeo{8nb}.
Recall that the nontrivial component~$\cB_{\bar{Y}} \subset \Db(\bar{Y})$ 
in the derived category of a cubic threefold~$\bar{Y}$ is defined by the semiorthogonal decomposition
\begin{equation}
\label{eq:cby}
\Db(\bar{Y}) = \langle \cB_{\bar{Y}}, \cO_{\bar{Y}}, \cO_{\bar{Y}}(\bar{H}) \rangle,
\end{equation}
where~$\bar{H}$ is the hyperplane class of~$\bar{Y} \subset \P^4$.

\begin{proposition}
\label{prop:cby-8}
If~$Y/\P^2$ is a conic bundle of type~\typeo{8nb},
$\cF \in \Ker(Y/\P^2)$ is the exceptional abstract spinor bundle on~$Y$ constructed in Lemma~\textup{\ref{lem:pf-cf1}},
and~$\cF^\perp$ is defined by~\eqref{eq:cfperp} then~$\cF^\perp \simeq \cB_{\bar{Y}}$,
where~$\cB_{\bar{Y}}$ is the component of the derived category of a smooth cubic threefold~$\bar{Y}$ defined by~\eqref{eq:cby}.
\end{proposition}

\begin{proof}
As before, it is enough to identify the orthogonal~$(\cF^0_{Y_\cF/\P^2})^\perp \subset \Ker(Y_\cF/\P^2)$ 
with~$\cB_{\bar{Y}}$.
Recall from Corollary~\ref{cor:yprime-12-10-8} that~$Y_\cF \cong \Bl_\ell(\bar{Y})$ 
is the blowup of a smooth cubic threefold~$\bar{Y}$ along a line~$\ell \subset \bar{Y}$.
Therefore, we have a semiorthogonal decomposition
\begin{equation*}
\Db(Y_\cF) = \langle \cB_{\bar{Y}}, \cO_{Y_\cF}, \cO_{Y_\cF}(\bar{H}), \cO_{E}, \cO_{E}(\bar{H}) \rangle,
\end{equation*}
where~$\bar{H}$ is the pullback to~$Y_\cF$ of the hyperplane class of~$\bar{Y}$ 
and~$E \subset Y_\cF$ is the exceptional divisor of the blowup~$Y_\cF \to \bar{Y}$.
Note that~$h = \bar{H} - E$ is the pullback of the line class with respect to the conic bundle morphism~$Y_\cF \to \P^2$.
Now we apply a sequence of mutations.

First, we mutate~$\cO_{E}$ to the left of~$\cO_{Y_\cF}$ and~$\cO_{E}(\bar{H})$ to the left of~$\cO_{Y_\cF}(\bar{H})$.
Since
\begin{equation*}
\Ext^\bullet(\cO_{Y_\cF}(\bar{H}), \cO_{E}) \cong \rH^\bullet(E, \cO_{E}(-\bar{H})) \cong \rH^\bullet(\ell, \cO_{\ell}(-1)) = 0
\end{equation*}
and
\begin{equation*}
\Ext^\bullet(\cO_{Y_\cF}(\bar{H}), \cO_{E}(\bar{H})) \cong 
\Ext^\bullet(\cO_{Y_\cF}, \cO_{E}) \cong 
\rH^\bullet(E, \cO_{E}) \cong \rH^\bullet(\ell, \cO_{\ell}) = \kk,
\end{equation*}
the results of the mutations (up to shift) are the line bundles~$\cO_{Y_\cF}(-E)$ 
and~$\cO_{Y_\cF}(\bar{H}-E) \cong \cO_{Y_\cF}(h)$, respectively,
and we obtain a semiorthogonal decomposition
\begin{equation*}
\Db(Y_\cF) = \langle \cB_{\bar{Y}}, \cO_{Y_\cF}(-E), \cO_{Y_\cF}, \cO_{Y_\cF}(h), \cO_{Y_\cF}(\bar{H}) \rangle.
\end{equation*}

Next, we mutate~$\cO_{Y_\cF}(\bar{H})$ to the far left.
Since~$K_{Y_\cF} = -2\bar{H} + E$ and~$\bar{H} + K_{Y_\cF} = -\bar{H} + E = -h$, we obtain
\begin{equation*}
\Db(Y_\cF) = \langle \cO_{Y_\cF}(-h), \cB_{\bar{Y}}, \cO_{Y_\cF}(-E), \cO_{Y_\cF}, \cO_{Y_\cF}(h) \rangle.
\end{equation*}

Finally, we mutate~$\cB_{\bar{Y}}$ and~$\cO_{Y_\cF}(-E)$ to the left of~$\cO_{Y_\cF}(-h)$.
Since~$K_{Y_\cF/\P^2} = \bar{H} -2E = h - E$, 
the mutation of~$\cO_{Y_\cF}(-E)$ is realized by the exact sequence
\begin{equation*}
0 \to \cO_{Y_\cF}(-E) \to \cF^0(-h) \to \cO_{Y_\cF}(-h) \to 0,
\end{equation*}
obtained from~\eqref{eq:csb} by a twist.
Thus, we obtain a semiorthogonal decomposition 
\begin{equation*}
\Db(Y_\cF) = \langle \cB_{\bar{Y}}, \cF^0(-h), \cO_{Y_\cF}(-h), \cO_{Y_\cF}, \cO_{Y_\cF}(h) \rangle.
\end{equation*}
It follows that~$\Ker(Y_\cF/\P^2) =\langle \cB_{\bar{Y}}, \cF^0(-h) \rangle$,
hence~$(\cF^0_{Y_\cF/\P^2})^\perp \simeq (\cF^0(-h))^\perp \simeq \cB_{\bar{Y}}$.
\end{proof}

\subsection{Categorical absorption for Fano threefolds}

In this subsection we show that the abstract spinor bundle~$\cF$ on~$Y$ constructed in Lemma~\ref{lem:pf-cf1} 
gives rise to a Mukai bundle on the corresponding 1-nodal Fano threefold~$X$
and using this, we construct a categorical absorption of singularities for~$X$.

\begin{proposition}
\label{prop:cf1}
Let~$X$ be a nonfactorial $1$-nodal Fano threefold of type~\typeo{12nb}, \typeo{10na}, or~\typeo{8nb},
let~$\pi \colon Y \to X$ be a small resolution of singularities that has a structure of a conic bundle,
and let~$\cF$ be the exceptional abstract spinor bundle on~$Y$ constructed in Lemma~\textup{\ref{lem:pf-cf1}}. 
Then 
\begin{equation}
\label{eq:cux}
\cU_X \coloneqq \pi_*(\cF(-h))
\end{equation}
is a $(-K_X)$-stable exceptional vector bundle on~$X$ such that
\begin{equation*}
\rank(\cU_X) = 2,
\qquad
\rc_1(\cU_X) = K_X,
\qquad
\rH^\bullet(X,\cU_X) = 0
\end{equation*}
and~$\cF \cong \pi^*\cU_X(h)$.
Finally, $\cU_X^\vee$ is globally generated with~$\rH^0(X, \cU_X^\vee) = \kk^{\oplus(k + 5)}$.
\end{proposition}

\begin{proof}
Recall the exceptional curve~$C \subset Y$ (see~\eqref{eq:curve-c}) of the contraction~$\pi \colon Y \to X$.
The isomorphism~\eqref{eq:cf1-c} implies that~$\cF(-h)\vert_C \cong \cO_C^{\oplus 2}$, 
hence~\eqref{eq:cux} defines a vector bundle of rank~$2$ such that~$\pi^*\cU_X \cong \cF(-h)$.

The bundle~$\cU_X$ is exceptional because~$\cF$ is (see Corollary~\ref{cor:cf-exc}) and~$\pi^*$ is fully faithful.
Using Lemma~\ref{lem:pf-cf1} and~\eqref{eq:ky} we find~$\pi^*(\rc_1(\cU_X)) = \rc_1(\cF(-h)) = -H = \pi^*K_X$, 
hence~$\rc_1(\cU_X) = K_X$.
Moreover, Lemma~\ref{lem:pf-cf1} also implies the vanishing of~$\rH^\bullet(Y, \cF(-h)) = 0$, 
hence~$\rH^\bullet(X, \cU_X) = 0$,
i.e., $\cU_X$ is acyclic.

Next, we check global generation.
Dualizing~\eqref{eq:def-cf1} and twisting it by~$\cO_Y(h)$, we obtain an exact sequence
\begin{equation}
\label{eq:cuv-ses}
0 \to \cO_Y(h) \to \pi^*\cU_X^\vee \to \cI_{C,Y}(H - h) \to 0.
\end{equation}
We claim that~$\pi_*\cO_Y(h)$ and~$\pi_*\cI_{C,Y}(H - h) = 0$ are pure globally generated sheaves.
Indeed, note that
\begin{equation*}
\pi_*\cO_Y(h) \cong \hat\pi_*\cO_{\hY}(h)
\qquad\text{and}\qquad 
\pi_*\cI_{C,Y}(H - h) \cong \hat\pi_*\cO_{\hY}(H - h - E),
\end{equation*}
where~$\hY \coloneqq \Bl_C(Y) \cong \Bl_{x_0}(X) \xrightarrow{\ \hat\pi\ } X$ is the blowup of the node~$x_0 = \pi(C) \in X$
and~$E \subset \hY$ is its exceptional divisor,
hence~$E \cong \P^1 \times \P^1$ and~$\cO_E(-E) \cong \cO_E(1,1)$.
It follows that~$\cO_E(h) \cong \cO_E(1,0)$ and~$\cO_E(H - h - E) \cong \cO_E(0,1)$,
hence~\cite[Lemma~6.3]{KS22} implies that
\begin{equation*}
\bR^{>0}\pi_*\cO_Y(h) = 
\bR^{>0}\hat\pi_*\cO_\hY(h) = 0
\qquad\text{and}\qquad 
\bR^{>0}\pi_*\cI_{C,Y}(H - h) =
\bR^{>0}\hat\pi_*\cO_\hY(H - h - E) = 0.
\end{equation*}
Therefore, $\pi_*\cO_Y(h)$ and~$\pi_*\cI_{C,Y}(H - h)$ are pure sheaves.

To prove that the first is globally generated, 
consider the  pullback along the morphism~$f \colon Y \to \P^2$ of the (twisted) Koszul complex of~$\P^2$:
\begin{equation*}
0 \to \cO_Y(-2h) \to \cO_Y(-h)^{\oplus 3} \to \cO_Y^{\oplus 3} \to \cO_Y(h) \to 0.
\end{equation*}
Since~$\bR^{>0}\pi_*\cO_Y(-h) = \bR^{>0}\hat\pi_*\cO_\hY(-h) = 0$ (again by~\cite[Lemma~6.3]{KS22}) 
and the dimension of any fiber of~$\pi$ is less than~$2$, 
we conclude that the morphism~$\cO_X^{\oplus 3} \cong \pi_*\cO_Y^{\oplus 3} \to \pi_*\cO_Y(h)$
is surjective, hence~$\pi_*\cO_Y(h)$ is globally generated.

Similarly, using the fact that~$h_+ \coloneqq H - h - E$ is base point free on~$\hY$ 
(indeed, by~\cite[Proposition~3.3]{KP23} it is isomorphic to the pullback of the ample generator of the Picard group
of a quadric~$Q^3$, or~$\P^3$, or~$\P^2$)
and~$\bR^{>0}\hat\pi_*\cO_Y(-h_+) = 0$ by~\cite[Lemma~6.3]{KS22}, 
we conclude that the sheaf~$\hat\pi_*\cO_Y(h_+) \cong \pi_*\cI_{C,Y}(H - h)$ is also globally generated.

Now, pushing forward~\eqref{eq:cuv-ses} 
and using global generation of~$\cO_Y(h)$ and~$\cI_{C,Y}(H - h)$ proved above 
together with the cohomology vanishing~$\rH^1(Y, \cO_Y(h)) = \rH^1(\P^2, \cO_{\P^2}(1)) = 0$,
we see that~$\cU_X^\vee$ is globally generated.
Since~$\rH^0(Y,\cO_Y(h)) \cong \kk^3$ 
and~$\rH^0(Y, \cI_{C,Y}(H - h)) \cong \rH^0(\P^2, \cE^\vee(-1)) \cong \kk^{\oplus(k+2)}$ by~\eqref{eq:cev},
we conclude that~$\rH^0(X, \cU_X^\vee)  \cong \kk^{\oplus(k+5)}$.

Finally, we prove $(-K_X)$-stability of~$\cU_X^\vee$.
By the argument of~\cite[Lemma~2.12]{KS23} it is enough to show that 
there are no nontrivial morphisms from~$\pi^*\cU_X^\vee$ to movable line bundles 
of smaller slope with respect to~$\pi^*(-K_X) = H$.
But the movable cone of~$Y$ is generated by the classes~$h$ and~$H - h$ by~\cite[Lemma~3.2]{KP23}, 
and it is easy to compute that 
\begin{equation*}
H^2 \cdot h =  k + 6
\qquad\text{and}\qquad 
H^2 \cdot (H - h) = (4k + 10) - (k + 6) = 3k + 4.
\end{equation*}
For~$k \in \{1,2,3\}$ these pairs of integers are~$(7,7)$, $(8,10)$, and~$(9,13)$, respectively.
Since~$\pi^*\cU_X^\vee$ is an extension~\eqref{eq:cuv-ses}, 
it follows easily that the only way to destabilize~$\pi^*\cU_X^\vee$ is by having a morphism from it to~$\cO_Y(h)$.
But such a morphism would split the sequence~\eqref{eq:cuv-ses}, 
in contradiction to its definition.
\end{proof}

\begin{remark}
\label{rem:cb-mb}
The properties of the bundle~$\cU_X$ proved in Proposition~\ref{prop:cf1}
mean that it is {\sf a Mukai bundle} on~$X$ in the sense of~\cite[Definition~1.2]{KS23}, see also~\cite[Definition~5.1]{BKM}.
\end{remark}

In the next theorem we construct a categorical absorption for nonfactorial 1-nodal Fano threefolds 
of type~\typeo{12nb}, \typeo{10na}, and~\typeo{8nb},
using the terminology and techniques developed in~\cite{KS22}.

\begin{theorem}
\label{thm:absorption}
Let~$X$ be a nonfactorial $1$-nodal Fano threefold of type~\typeo{12nb}, \typeo{10na}, or~\typeo{8nb}, 
and let~$\pi \colon Y \to X$ be a small resolution of singularities 
that has a structure of a conic bundle~$f \colon Y \to \P^2$. 
There is a semiorthogonal decomposition
\begin{equation}
\label{eq:dbx-absorption}
\Db(X) = \langle \rP_X, \cA_X, \cU_X, \cO_X \rangle,
\end{equation}
where~$\rP_X$ is a $\Pinfty{2}$-object providing a universal deformation absorption of singularities of~$X$,
$\cU_X$ is the Mukai bundle defined in~\eqref{eq:cux},
and~$\cA_X \subset \Db(X)$ is a smooth and proper admissible subcategory such that
\begin{equation}
\label{eq:cax}
\cA_X \simeq 
\begin{cases}
\Db(\Qu_3), & \text{for type~\typeo{12nb},}\\
\Db(\Gamma_2), & \text{for type~\typeo{10na},}\\
\cB_{\bar{Y}}, & \text{for type~\typeo{8nb},}
\end{cases}
\end{equation}
where recall that~$\Qu_3$ is the $3$-Kronecker quiver with two vertices and three arrows,
$\Gamma_2$ is a curve of genus~$2$, 
$\bar{Y}$ is a smooth cubic threefold, 
and~$\cB_{\bar{Y}}$ is the component of~$\Db(\bar{Y})$ defined by~\eqref{eq:cby}.
\end{theorem}

\begin{proof}
Recall that the exceptional curve of the contraction~$\pi$ is the curve~$C$ defined by~\eqref{eq:curve-c},
hence the kernel of the pushforward functor~$\pi_* \colon \Db(Y) \to \Db(X)$ is generated by the spherical object~$\cO_C(-1)$, 
see~\cite[Theorem~5.8 and Corollary~5.10]{KS22}.
Moreover, using~\eqref{eq:h-c} we compute
\begin{equation*}
\Ext^\bullet(\cO_{Y}(h - H), \cO_C(-1)) \cong 
\rH^\bullet(C, \cO_Y(H - h) \otimes \cO_C(-1)) \cong
\rH^\bullet(C, \cO_C(-2)) \cong
\kk[-1],
\end{equation*}
hence the line bundle~$\cO_Y(h - H)$ is adherent to~$\cO_C(-1)$ in the sense of~\cite[Definition~3.9]{KS22}.
Applying~\cite[Theorem~6.17 and Corollary~6.18]{KS22} we obtain semiorthogonal decompositions
\begin{equation}
\label{eq:dby}
\Db(Y) = \langle \cO_Y(h - H), \bT_{\cO_C(-1)}(\cO_Y(h - H)), \pi^*(\cD) \rangle
\qquad\text{and}\qquad 
\Db(X) = \langle \rP_X, \cD \rangle,
\end{equation}
where~$\bT_{\cO_C(-1)}$ is the spherical twist with respect to~$\cO_C(-1)$ 
and~$\rP_X = \pi_*\cO_Y(h - H)$ is a $\P^{\infty,2}$-object 
which provides a universal deformation absorption of singularities for~$X$.
It remains to describe~$\cD$.

First, note that~$\Ext^\bullet(\cO_C(-1), \cO_Y(h - H)) \cong \kk[-2]$ by the above computation and Serre duality,
hence the spherical twist~$\bT_{\cO_C(-1)}(\cO_Y(h - H))$ is defined by the distinguished triangle
\begin{equation}
\label{eq:tco}
\cO_C(-1)[-2] \to \cO_Y(h - H) \to \bT_{\cO_C(-1)}(\cO_Y(h - H)).
\end{equation}
We have~$\cO_C(-1), \cO_Y(h - H) \in \cO_Y^\perp$ 
(the first is obvious and the second follows from~\eqref{eq:ky}), hence~$\cO_X \in \cD$.

Similarly, we have~$\cO_C(-1) \in (\pi^*\cU_X)^\perp$
(because~$\pi_*(\cO_C(-1)) = 0$) and
\begin{equation*}
\Ext^\bullet(\pi^*\cU_X, \cO_Y(h - H)) \cong 
\Ext^{3 - \bullet}(\cO_Y(h), \pi^*\cU_X)^\vee \cong
\rH^\bullet(Y, \pi^*\cU_X(-h))^\vee \cong
\rH^\bullet(Y, \cF(-2h))^\vee = 0,
\end{equation*}
where the first is Serre duality, the second is obvious,
the third follows from the definition of~$\cU_X$ (see Proposition~\ref{prop:cf1}),
and the fourth is proved in Lemma~\ref{lem:pf-cf1}.
Thus, $\cU_X \in \cD$, and since~$\cU_X$ is exceptional and acyclic, 
we obtain~\eqref{eq:dbx-absorption}.

It remains to identify the component~$\cA_X \subset \Db(X)$ defined by~\eqref{eq:dbx-absorption}
(or, equivalently, the subcategory~$\pi^*(\cA_X) \subset \Db(Y)$).
For this we consider decompositions~\eqref{eq:sod-y-cb}, which we rewrite as
\begin{equation*}
\Db(Y) = \langle (\cF(-h))^\perp, \cF(-h), \cO_Y(-h), \cO_Y, \cO_Y(h) \rangle,
\end{equation*}
where~$(\cF(-h))^\perp$ is the orthogonal to~$\cF(-h)$ in~$\Ker(Y/\P^2)$.
Now we apply a sequence of mutations.

First, we mutate~$\cO_Y(h)$ to the far left.
Since~$K_Y = -H$, we obtain
\begin{equation}
\label{eq:dby-2}
\Db(Y) = \langle \cO_Y(h - H), (\cF(-h))^\perp, \cF(-h), \cO_Y(-h), \cO_Y \rangle.
\end{equation}

Next, we mutate~$\cO_Y(-h)$ to the left of~$\cF(-h)$.
Since by Corollary~\ref{cor:cf-exc} we have
\begin{equation*}
\Ext^\bullet(\cF(-h), \cO_Y(-h)) \cong
\rH^\bullet(Y, \cF^\vee) = \kk,
\end{equation*}
the mutation is isomorphic (up to shift) to the cone of the unique nontrivial morphism~\mbox{$\cF(-h) \to \cO_Y(-h)$}.
Then the exact sequence
\begin{equation}
\label{eq:cf1-twisted}
0 \to \cO_Y(h - H) \to \cF(-h) \to \cO_Y(-h) \to \cO_C(-1) \to 0
\end{equation}
obtained from~\eqref{eq:def-cf1} by a twist
identifies the cone of~\mbox{$\cF(-h) \to \cO_Y(-h)$}
with the shifted cone of the morphism~\mbox{$\cO_C(-1)[-2] \to \cO_Y(h - H)$},
and now~\eqref{eq:tco} shows that the result is~$\bT_{\cO_C(-1)}(\cO_Y(h - H))$.
Thus, we obtain
\begin{equation*}
\Db(Y) = \langle \cO_Y(h - H), (\cF(-h))^\perp, \bT_{\cO_C(-1)}(\cO_Y(h - H)), \cF(-h), \cO_Y \rangle.
\end{equation*}
Now we recall that~$\cF(-h) \cong \pi^*\cU_X$ by Proposition~\ref{prop:cf1}, 
and comparing the above decomposition with the definition of~$\cA_X$ in~\eqref{eq:dbx-absorption} and~\eqref{eq:dby},
we obtain an equivalence
\begin{equation*}
\cA_X \simeq (\cF(-h))^\perp,
\end{equation*}
given by the mutation functors with respect to~$\bT_{\cO_C(-1)}(\cO_Y(h - H))$.
Finally, combining this equivalence with the description of~$(\cF(-h))^\perp \simeq \cF^\perp$
given in Propositions~\ref{prop:cby-12}, \ref{prop:cby-10}, and~\ref{prop:cby-8},
we deduce~\eqref{eq:cax}. 
\end{proof}

Applying~\cite[Theorem~1.8]{KS22} and the argument of~\cite[Theorem~3.6]{KS23}, we obtain

\begin{corollary}
\label{cor:def-abs}
If\/~$\cX \to B$ is a smoothing of a nonfactorial~$1$-nodal Fano threefold~$X$ 
of type~\typeo{12nb}, \typeo{10na}, or~\typeo{8nb} over a smooth punctured curve~$(B,o)$
then there is a smooth and proper family of triangulated categories~$\cA$ over~$B$ such that~$\cA_o \simeq \cA_X$
and for each point~$b \ne o$ in~$B$ there is an equivalence~$\cA_b \simeq \cA_{\cX_b}$,
where~$\cA_{\cX_b} = \langle \cU_{\cX_b}, \cO_{\cX_b} \rangle ^\perp \subset \Db(\cX_b)$ 
is the nontrivial component of the derived category 
of the smooth prime Fano threefold~$\cX_b$ of genus~$12$, $10$, or~$8$, respectively.
\end{corollary}

Recall from~\cite[Theorems~4.1, 4.5, 4.7]{K09} that
\begin{equation*}
\cA_{\cX_b} \simeq 
\begin{cases}
\Db(\Qu_3), & \text{if~$\g(\cX_b) = 12$,}\\
\Db(\Gamma_{2,b}), & \text{if~$\g(\cX_b) = 10$,}\\
\cB_{\bar{Y}_b}, & \text{if~$\g(\cX_b) = 8$,}
\end{cases}
\end{equation*}
where~$\Gamma_{2,b}$ and~$\bar{Y}_b$ is the curve of genus~2 and the smooth cubic threefold 
associated with the smooth prime Fano threefold~$\cX_b$ of genus~10 or~8, respectively.
In particular, the family of categories~$\cA$ is isotrivial for type~\typeo{12nb}.
We expect that for types~\typeo{10na} and~\typeo{8nb}
the smoothing~$\cX$ also can be chosen in such a way that~$\cA$ is isotrivial.

\subsection{Conic bundles of type~\typeo{5n}}
\label{ss:5n}

In this final subsection we discuss conic bundles~$f \colon Y \to \P^2$ 
that provide small resolutions of singularities for nonfactorial 1-nodal Fano threefolds~$X$ of type~\typeo{5n}
in the notation of~\cite{KP23}.
It follows from~\cite[Remark~6.6]{KP23} that the corresponding quadratic forms 
can be written as~$q \colon \cL \to \Sym^2\cE^\vee$, where
\begin{equation}
\label{eq:cev-5n}
\cL \cong \cO_{\P^2}(-1)
\qquad\text{and}\qquad 
\cE \cong \cO_{\P^2} \oplus \cO_{\P^2}(-1) \oplus \cO_{\P^2}(-1).
\end{equation}
In particular, the section of~$\P_{\P^2}(\cE) \to \P^2$ corresponding to the summand~$\cO_{\P^2} \subset \cE$
intersects~$Y$ along a curve~$C \subset Y$ such that~$f\vert_C$ is an isomorphism~$C \to L$ onto a line~$L \subset \P^2$.
Moreover, \eqref{eq:ky} and~\eqref{eq:h-c} still hold, 
the curve~$C$ is the exceptional locus of the small contraction~$\pi \colon Y \to X$,
and~\eqref{eq:cn-c-y} holds.

The discriminant divisor of the conic bundle~$Y/\P^2$ has degree~$7$, 
and it contains the line~$L$ if and only if~$X$ has a 1-dimensional family of anticanonical lines through the node;
this happens in the situation described in~\cite[Remark~1.5]{KP23}.

For conic bundles of type~\typeo{5n} the construction of Lemma~\ref{lem:pf-cf1} 
also produces an abstract spinor bundle~$\cF$.
However, the computation of Proposition~\ref{prop:good-cb0} shows that in this case
\begin{equation*}
f_*\cEnd^0(\cF) \cong \cO_{\P^2}(-3) \oplus \cO_{\P^2}(-2) \oplus \cO_{\P^2}(-2),
\end{equation*}
hence~$\Ext^\bullet(\cF, \cF) \cong \kk \oplus \kk[-2]$, and therefore~$\cF$ is not exceptional.
It also follows that if~$q_\cF \colon \cL_\cF \to \Sym^2\cE_\cF^\vee$ 
is the quadratic form of the spinor modification~$Y_\cF/\P^2$ then
\begin{equation*}
\cL_\cF \cong \cO_{\P^2}(-1)
\qquad\text{and}\qquad 
\cE_\cF \cong \cO_{\P^2} \oplus \cO_{\P^2}(-1) \oplus \cO_{\P^2}(-1).
\end{equation*}
In particular, $q_\cF$ has the same form as~$q$.
In fact, the quadratic form~$q_\cF$ coincides with~$q$.

\begin{lemma}
If\/~$Y/\P^2$ is a conic bundle of type~\typeo{5n},
the abstract spinor bundle~$\cF$ defined in Lemma~\textup{\ref{lem:pf-cf1}}
coincides with the canonical spinor bundle~$\cF^{-1}$ defined in~\eqref{eq:csb-1}. 
In particular, in this case~$Y_\cF \cong Y$.
\end{lemma}

\begin{proof}
The description of the curve~$C$ given above implies that~$C$ coincides with the intersection of the zero loci
of the two global sections of the line bundle~$\cO_Y(H-h)$.
Therefore, the ideal sheaf~$\cI_{C,Y}$ has a Koszul resolution
\begin{equation*}
0 \to \cO_Y(2h - 2H) \to \cO_Y(h - H)^{\oplus 2} \to \cI_{C,Y} \to 0.
\end{equation*}
It is easy to check that~\eqref{eq:def-cf1} is obtained from this by a pushout,
hence there is an exact sequence
\begin{equation*}
0 \to \cO_Y(2h - 2H) \to \cO_Y(h - H)^{\oplus 2} \oplus \cO_Y(2h - H) \to \cF \to 0.
\end{equation*}
Note that the middle term of this exact sequence can be rewritten as~$f^*\cE(2h - H)$.
Therefore, dualizing and twisting this sequence by~$\cO_Y(2h - H)$, 
and using an isomorphism~$\cF^\vee(2h - H) \cong \cF$ that follows from~$\rc_1(\cF) = 2h - H$,
we obtain an exact sequence
\begin{equation*}
0 \to \cF \to f^*\cE^\vee \to \cO_Y(H) \to 0.
\end{equation*}
Since~$\cF$ is an abstract spinor bundle, comparing this sequence with~\eqref{eq:csb-1} 
and using the uniqueness claim of Lemma~\ref{lem:cf0-cf1}, 
we conclude that~$\cF \cong \cF^{-1}$, as required.
\end{proof}

Furthermore, the argument of Proposition~\ref{prop:cf1} shows that~$\cU_X \coloneqq \pi_*(\cF(-h))$
is still an acyclic vector bundle of rank~2 on~$X$ with~$\rc_1(\cU_X) = K_X$,
however it is neither exceptional (because~$\cF$ is not), 
nor $(-K_X)$-stable (it is destabilized by the sequence~\eqref{eq:cuv-ses}).
On the other hand, $\cU_X^\vee$ is still globally generated 
and induces an embedding
\begin{equation*}
X \hookrightarrow \Cone(\P^1 \times \P^2) \subset \Gr(2,5)
\end{equation*}
onto a Weil divisor in a codimension~2 Schubert subvariety of~$\Gr(2,5)$.

Moreover, the argument of Theorem~\ref{thm:absorption} shows that there is a semiorthogonal decomposition
\begin{equation*}
\Db(X) = \langle \rP_X, \cA_X, \cO_X \rangle,
\end{equation*}
where~$\rP_X$ is a $\Pinfty{2}$-object providing a universal deformation absorption of~$X$
and~$\cA_X \subset \Db(X)$ is a smooth and proper admissible subcategory.
The relation of~$\cA_X$ to the kernel category~$\Ker(Y/\P^2)$ of the original conic bundle is not that clear.
The semiorthogonal decompositions
\begin{alignat*}{2}
\Db(Y) &= \langle \cO_Y(h - H),\quad  \bT_{\cO_C(-1)}(\cO_Y(h - H)), \pi^*(\cA_X), \quad & \cO_Y \rangle,
\\
\Db(Y) &= \langle \cO_Y(h - H),\quad  \Ker(Y/\P^2), \cO_Y(-h), & \cO_Y \rangle\hphantom{,}
\end{alignat*}
obtained similarly to~\eqref{eq:dby} and~\eqref{eq:dby-2} imply that~$\cA_X$ and~$\Ker(Y/\P^2)$
are~\emph{Krull--Schmidt partners} in the sense of~\cite{O16-KS}.
However, it is not clear if they are equivalent or not.

Finally, the analogue of Corollary~\ref{cor:def-abs} shows that the category~$\cA_X$
deforms into the component~$\cA_{\cX_b}$ of the derived category of a smooth prime Fano threefold~$\cX_b$ of genus~5.
The latter threefold is isomorphic to the intersection of three quadrics in~$\P^6$,
which gives an equivalence
\begin{equation*}
\cA_{\cX_b} \simeq \Ker(\cQ_b/\P^2),
\end{equation*}
where~$\cQ_b \to \P^2$ is a conic bundle 
obtained from the family of 5-dimensional quadrics in~$\P^6$ through~$X$
by the hyperbolic reduction with respect to a line in~$X$ 
(cf.\ the construction in~\cite[\S3.2]{KS18}).
The conic bundles~$Y \to \P^2$ and~$\cQ_b \to \P^2$ both have discriminant curves of degree~7,
however, while the latter conic bundle is produced from a \emph{nondegenerate} (even) theta-characteristic on such a curve, 
the former corresponds to a \emph{degenerate even} theta-characteristic.

\end{document}